\documentclass{article}
\usepackage{commands}
\usepackage{layout}

\title{Optimal Control of Nonautonomous Saddle Node Bifurcations}
\author{Christopher Beekmann\thanks{Department of Mathematics and Computer Science, Eindhoven University of Technology, De Groene Loper 5, 5612 AE Eindhoven, The Netherlands.}\ \thanks{c.a.j.beekmann@tue.nl.}  , \  Kerstin Lux-Gottschalk\thanks{Department of Mathematics and Computer Science, Eindhoven University of Technology, De Groene Loper 5, 5612 AE Eindhoven, The Netherlands.}\ \thanks{k.m.lux@tue.nl.}}

\begin{document}

\date{}
\maketitle

\begin{center}
\section*{Abstract}
\end{center}
    Nonautonomous saddle node bifurcations commonly appear in mathematical models, however, an optimal control theory tailored to such systems is not yet well established. We present conceptual work on the nonautonomous saddle node normal form. We develop an approach that reformulates the optimal control problem into an unconstrained minimization problem by exploiting a connection between bifurcation and rate-induced overshoots. The key component is a rate function that assigns the critical rate to every nonautonomous forcing. We show that in the saddle node normal form, the critical rate is unique, and hence the rate function is well defined. 
    While this reformulation is valid in any case, the rate function is often unknown. To resolve this, we additionally present an approximation of the rate function that leads to a closed form expression of the approximated control.
    This makes the approach appealing in situations, in which we need repeated solves of the control problem.

\clearpage 

\section{Introduction} \label{sec::introduction}
Differential equations are powerful tools to model the evolution of complex real-world systems. Oftentimes, these dynamics are highly nonlinear and interesting phenomena can appear such as tipping behaviour of the system, i.e.~sudden, drastic and often irreversible changes in a systems dynamics upon small changes in its input. These transitions are also called critical transitions and commonly occur in climate models (\cite{HumanSystems23, TippingElementsinEarthsclimateSystem, inverseSquareLaw19}), neuronal models (\cite{Ashwin_neuroscience2016, avitabile22, kirillov15, Baesens_2013}), ecological models (\cite{HumanSystems23, WieczorekZombieFires, liquidcrystals13}) and even laser systems (\cite{semiconductorlasers}). In climate models, tipping can for example correspond to a collapse of an ocean circulation such as the Atlantic Meridional Overturning Circulation. In ecological models, the collapse of a population can be induced by crossing a tipping point and in neuronal models, this mechanism is often used to model the creation of an action potential.
\\
A rigorous mathematical theory for tipping phenomena has been developed (\cite{universal_explosive_phenomena}) and recently a lot of research has addressed different types of tipping: bifurcation-, rate-, and noise-induced tipping, the construction and analysis of early warning signs for tipping points as well as the implications of overshooting tipping thresholds (see e.g.\ \cite{Lux24, ahswin12Tippinginopensystems, Ashwin19_RateInducedThresholdForAMOC, Ashwin2017, rat_induced_applications_ecology, ditlevsen, non_markovian_scaling_laws, DUENAS2023138, boers26}).
\\
\\
Very often critical transitions can be explained by the occurrence of a (nonautonomous) saddle node bifurcation (e.g.~\cite{HumanSystems23, inverseSquareLaw19, Ashwin_neuroscience2016, avitabile22, WieczorekZombieFires, Lux24, Ashwin19_RateInducedThresholdForAMOC, Ashwin2017, rat_induced_applications_ecology, ditlevsen}), i.e.~the collision and subsequent annihilation of an attractor-repeller pair. Preventing this bifurcation is an important area of active research and of utmost importance for many applications such as climate systems (\cite{global_tipping_point_report}) or neuronal systems (\cite{Rajabi25, Wilson22, Li_neuron}). 
\\
\\
However, whereas previous research has already addressed tipping thresholds and overshoot scenarios, much less
has been done in terms of how to actually control such saddle node systems in a cost efficient way to
prevent critical transitions.

The authors of \cite{Rajabi25, Wilson22, Li_neuron} solve their optimal control problems using powerful tools like Euler-Lagrange or Hamilton-Jacobi-Bellman equations, or by employing the Pontryagin maximum principle (\cite{optimalControl}). Other approaches include discretizing the dynamics and parsing them as constraints to a nonlinear problem solver, e.g.~by using CasADi (\cite{Casadi}). However, all these methods tend to ignore the underlying saddle node structure that is present in the models, and as a result become increasingly harder to solve,
the more complex the model becomes.
We see great potential to improve this scaling behaviour,
since the underlying saddle node bifurcation structure remains the same, even in more complicated, higher-
dimensional models. Hence, an optimal control theory building on this structure has the potential to scale
better with the complexity of the surrounding dynamics.
\\
\\
In this paper, we make a first step towards such an optimal control theory, by developing it
for the simple case of the nonautonomous saddle node bifurcation normal form. The dynamics of this equation have been
studied extensively, for example in \cite{Longo_2024CriticalTransitionsForCaratheodory, QuadraticOde21, PiecewiseUniformlyContinuous24, Anagnostopoulou2012}.
The goal is to develop a formalism that provides more flexibility and better generalization to complex models by leveraging the underlying generic bifurcation structure. In particular, we show that one can reformulate the optimal control problem to an unconstrained minimization problem using information about the systems dynamics.  
\\
\\
In Section \ref{sec::problem_setup}, we introduce the mathematical problem setup by formulating the constrained optimal control problem \eqref{prob::OCP} subject to the dynamics of a nonautonomous saddle node bifurcation normal form. Further, we already showcase the reformulated problem \eqref{prob::RCP}. 
In Section \ref{sec::nonautonomous_fold}, we characterise the structural properties of the nonautonomous saddle node normal form that are needed to reformulate \eqref{prob::OCP}.
In particular, in Section \ref{sec::nonautonomoustheory}, we reiterate necessary concepts of nonautonomous dynamics, while in Section \ref{sec::saddlenode} some important results about the bifurcation behaviour of nonautonomous saddle node bifurcations, mostly based on \cite{Longo_2024CriticalTransitionsForCaratheodory,PiecewiseUniformlyContinuous24,QuadraticOde21}, are summarized and put in the context of the specific case of the nonautonomous saddle node bifurcation normal form. Finally, in Section \ref{sec::ratefunction}, we show the existence of a rate function that, for any nonautonomous input, distinguishes the dynamical possibilities of tipping and stabilizing the system.
In Section \ref{sec::reformOCproblem}, we derive the unconstrained reformulation of the \eqref{prob::OCP} based on the newly defined rate function. This reformulation is valid in any case, however, the rate function is often unknown. Therefore, in Section \ref{sec::analyticalsol}, by approximating the rate function, we show how to obtain a closed form expression for the optimal control of the reformulated \eqref{prob::OCP}. This makes an approximated control quickly and readily available, which is particularly appealing in the case where repeated solves of the \eqref{prob::OCP} are necessary (e.g.\ in case of uncertain parameters). In several numerical experiments inspired by climate and neurosciences applications, we compare the obtained approximated control with the numerical solution obtained by passing the problem as a constrained nonlinear optimization problem to IPOPT (\cite{ipopt}), a software package to solve nonlinear constrained optimization problems.

\subsection*{Notation}
Here, we summarize some notational conventions that are frequently used throughout this paper:
\begin{itemize}
    \item We write $l$ for the Lebesgue measure on $\R$, and $l$-a.a./$l$-a.e. for Lebesgue almost all and Lebesgue almost everywhere
    \item $\R^+ \coloneqq [0,\infty)$;
    \item Let $\Omega\subset\R$, we write $\SET(\Omega) \coloneqq L^1(\Omega) \cap L^\infty(\Omega)$. Note that $h\in \SET(\Omega)$ implies $h\in L^p(\Omega)$ for all $p\in [1,\infty]$;
    \item Let $\Omega\subset\R$, $BV(\Omega)$ is the space of functions with bounded variation on $\Omega$;
    \item Let $\Omega\subset\R$, $C_{bu}(\Omega)$ is the space of bounded and uniformly continuous functions on $\Omega$;
    \item Let $\Omega \subset \R$ and $f:\R\to\R$. We denote for $p\in[1,\infty)$
        $$
            ||f||_{p;\Omega} \coloneqq \left(\int_\Omega |f(x)|^p \dd x \right)^{\frac{1}{p}}
        $$
    and further
    $$
        ||f||_{\infty;\Omega} \coloneqq \esssup_{x\in\Omega} |f(x)|;
    $$
    \item $\indicator_{[t_1,t_2]}(t)$ denotes the indicator function, i.e.~
    \begin{align*}
        \indicator_{[t_1,t_2]}(t) = \left\{\begin{array}{cc}
           1  & t_1\leq t\leq t_2 \\
           0  & \text{otherwise;}
        \end{array}\right. 
    \end{align*}
    \item For a differentiable function $f: \R \times \R \to \R$, we write $f_t$ for its partial derivative with respect to the first component and $f_x$ as the partial derivative with respect to the second component.
    \item Whenever we want to emphasize the dependence of a solution to an ordinary differential equation (ODE) $x$ on parameters $p_1,p_2,...$, we write $x_{p_1,p_2,...}$. We do the same for equations, for example, if we want to highlight the dependence of equation (A.B) on $\lambda$ we write (A.B)$_\lambda$.
\end{itemize}

\clearpage

\section{Problem setup}
\label{sec::problem_setup}
Let $\lambda >0$. We assume an autonomous saddle node normal form $\dot{x} = \lambda -x^2$ that is destabilized by an external forcing $\alpha(t)\in\SET(\R^+) \coloneqq L^1(\R^+) \cap L^\infty(\R^+)$, resulting in the nonautonomous saddle node normal form
\begin{align}\label{prob::sec1uncontrolled}
    \dot{x}(t) = \lambda - \alpha(t) -x^2(t) \qquad x(0)= \mu^x >0.
\end{align}
This forcing leads to tipping (\cite{ahswin12Tippinginopensystems}), which in the saddle node normal form materializes as vanishing attractors and divergence of trajectories to negative infinity in finite time. We assume that tipping is undesirable for the saddle node system. This motivates the introduction of a control $u\in\SET(\R^+)$ into the system as a preventive measure. Hence, one obtains the \textit{controlled equation} 
\begin{align}\label{prob::sec1controlled}
    \dot{x}(t) = \lambda - \alpha(t) + u(t) -x^2(t).
\end{align}
Further, the control comes at a cost which is to be minimized. In the following, we aim to minimize the squared $L^2$ norm of $u$, i.e.~$\int_0^\infty u^2(t) \dd t$. Note, that we do this solely for simplicity and throughout Sections \ref{sec::problem_setup}-\ref{sec::reformOCproblem} this cost function could be replaced by any other that only depends on $u$, as long as it is reasonable enough such that solutions to the resulting optimal control problem still exist. We obtain the optimal control problem

\begin{problem}[Optimal Control Problem]
    Let $\lambda >0$ and $\alpha\in\SET(\R^+)$ such that the uncontrolled equations \eqref{prob::sec1uncontrolled} satisfies: $x \to -\infty$ in finite time. We consider the optimal control problem
    \begin{align}\tag{\textbf{OCP}}\label{prob::OCP}
        \min_{u \in \SET(\R^+)} \int_0^\infty u^2(t) \dd t \qquad s.t.~
        \left\{ 
            \begin{array}{llll}
            &\dot{x}(t) = \lambda - \alpha(t) + u(t) - x^2(t)  \\
            & x(0) = \mu^x >0 \\
            &\exists \  M \in \N \ s.t.~\ x(t) \geq -M \quad \forall \  t \in \R.
            \end{array}
        \right.
    \end{align}
\end{problem}
Here, the third constraint ensures that no bifurcation induced tipping can occur.
Note, since $\alpha$ is fixed we can define the \textit{effective forcing} $g\coloneqq \alpha - u$ and equivalently minimize $\int_0^\infty [\alpha(t) -g(t)]^2 \dd t$ w.r.t.~$g$ using the same constraints.  
Our approach consists of constructing a rate function $\Kappa(\lambda, \alpha ,g)$ that time-rescales $g$ such that the optimal control problem \eqref{prob::OCP} (see below) can alternatively be formulated as the unconstrained minimization problem
\begin{align}\tag{\textbf{RCP}}\label{prob::sec1::unofficialreform}
    \min_{g \in \SET(\R^+)} \int_0^\infty [\alpha(t) - g(t\Kappa(\lambda,\alpha,g))]^2 \dd t.
\end{align}
An advantage of this reformulated control problem \eqref{prob::sec1::unofficialreform} is that all dynamical properties of the system are contained in $\Kappa$ and therefore are decoupled from the minimization problem. This enables the use of powerful standard unconstrained optimization methods such as Newton-like methods. In contrast, \eqref{prob::OCP} is a minimization problem with constraints on the dynamical evolution. 
However, this comes at the cost that the rate function $\Kappa$ is not a priori known and needs to be calculated for every effective forcing $g$. Approaches to do so numerically are, for example, the power series in \cite{Kuehn2022} (with adaptions), the time compactification \cite{Tipping_thresholdsWieczorek_2023} or simply a Taylor expansion. Another benefit of the reformulation is that \eqref{prob::sec1::unofficialreform} naturally allows for a balance between computational efficiency and precision, without loss of numerical accuracy, by changing the way $\Kappa$ is computed. In particular, it is possible to find simple expressions that approximate $\Kappa$ and nonetheless keep trajectories bounded. One such expression is presented in Section \ref{sec::analyticalsol}.
\\\\
Note that the rate function $\Kappa$ is closely connected to overshooting (e.g.~\cite{inverseSquareLaw19}) and to a mixture of rate-induced tipping, a phenomenon that has been intensively studied in recent literature (e.g.~\cite{Tipping_thresholdsWieczorek_2023, Longo_2024CriticalTransitionsForCaratheodory, QuadraticOde21, PiecewiseUniformlyContinuous24, Ashwin2017, ahswin12Tippinginopensystems, Ashwin19_RateInducedThresholdForAMOC}), and size-induced tipping (e.g~\cite{Longo_2024CriticalTransitionsForCaratheodory}). In fact, it assigns every forcing exactly the critical rate. In particular, our approach to reformulating the optimal control problem \eqref{prob::OCP} can be thought of as reformulating a bifurcation-induced tipping problem (the system crosses the bifurcation point) to a rate-induced overshooting problem.

\clearpage

\section{Nonautonomous saddle node bifurcation} \label{sec::nonautonomous_fold}
The study of nonautonomous saddle node bifurcation (as for example done in \cite{QuadraticOde21}) relies on the analysis of the corresponding skew-product flow. As such, the base flow is often assumed to be invertible, which in particular means that \eqref{prob::sec1controlled} must be defined for all $t\in\R$. This contrasts \eqref{prob::OCP}, where \eqref{prob::sec1uncontrolled} is a priori only defined for $t\geq 0$. The details of how this gap is bridged are discussed in Section \ref{sec::reformOCproblem}, for now we simply note that the dynamics of \eqref{prob::sec1controlled} must be extended to $(-\infty,0)$ and we do so by considering

\begin{center}    
\begin{tcolorbox}[width=0.9\linewidth, colback=gray!10, colframe=gray!40, boxrule=0.5pt, arc=3mm, left=2mm, right=2mm, top=0mm, bottom=3mm, halign=center, valign=center]
\begin{align}\tag{x}\label{eqq::x}
    \Dot{x}(t) = \lambda - h(t;\kappa) - x(t)^2,
\end{align}
\end{tcolorbox}
\end{center}

where $\lambda, \kappa >0$ and 
\begin{align}\label{eqq::defh}
    h(t;\kappa) \coloneqq \left\{\begin{array}{cc}
        \alpha(t) & t<0 \\
         g(t\kappa) & t\geq0 
    \end{array}\right.
\end{align}
for some $\alpha\in \SET(\R)= L^1(\R)\cap L^\infty(\R), \ g \in \SET(\R^+)$. Here, $g$ incorporates the control by $g = \alpha - u$. Furthermore, one can intuitively think about this as an autonomous fold in equilibrium that has been destabilized on $(-\infty,0)$ by a forcing $\alpha$ and arrives at $\mu^x$ at time $t=0$.
Note that $h(\cdot;\kappa)\in \SET(\R)$ for any $\kappa>0$. We show later that by choosing $\kappa>0$ correctly, one can ensure the boundedness of the attractor of \eqref{eqq::x} for any $g\in \SET(\R^+)$. 
\\\\
This section is devoted to the analysis of \eqref{eqq::x} and in particular the behaviour of the pullback attractor upon variation of $\kappa$. We rely on the works \cite{Longo_2024CriticalTransitionsForCaratheodory, QuadraticOde21, PiecewiseUniformlyContinuous24} that establish that the pullback attractor of \eqref{eqq::x} is in one of three Cases:
\begin{itemize}
    \item \textbf{Case A}: It is hyperbolic (cf.~Section \ref{sec::nonautonomoustheory}) and connects from $\sqrt{\lambda}$ at $-\infty$ to $\sqrt{\lambda}$ at $\infty$;
    \item \textbf{Case B}: It is not hyperbolic but bounded and connects from $\sqrt{\lambda}$ at $-\infty$ to $-\sqrt{\lambda}$ at $\infty$;
    \item \textbf{Case C}: It is not bounded and exists only on an open halfline $(-\infty,t_f)$.
\end{itemize}

Section 2 is structured as follows: In Section \ref{sec::nonautonomoustheory} we reiterate some necessary basics of nonautonomous dynamics, such as exponential dichotomies and persistence of hyperbolic solutions under perturbations. We stick to the resources \cite{Longo_2024CriticalTransitionsForCaratheodory, QuadraticOde21, PiecewiseUniformlyContinuous24} as they are best adapted to our case, however, most results in that section can also be found in e.g.~\cite{NonautonomousDynamicalSystems} (and references therein). Section \ref{sec::saddlenode} is dedicated to summarizing the results of \cite{Longo_2024CriticalTransitionsForCaratheodory, QuadraticOde21, PiecewiseUniformlyContinuous24} regarding scalar quadratic or coercive and concave ordinary differential equations and adjusting their results to our needs.

Finally, in Section \ref{sec::ratefunction}, we analyse how \eqref{eqq::x} changes with variation of $\kappa$. We establish that for small $\kappa>0$, \eqref{eqq::x} is in Case C, for large $\kappa>0$ it is in Case A and there is a unique $\kappa_{c} >0$, for which \eqref{eqq::x} is in Case B. In particular, there is a unique transversal transition from A to C under variation of $\kappa$. Transversal in this context means that the transition interval, where the dynamics are in Case B, is a singleton. This result is, to the best of our knowledge, new in the literature and is a key component of the reformulation in Section \ref{sec::reformOCproblem}. This then leads to the definition of a rate function $\Kappa$ assigning every $h$ as in \eqref{eqq::defh} a rate such that \eqref{eqq::x}$_{h(\cdot;\Kappa)}$ is in Case B.

In preparation for Section \ref{sec::ratefunction}, we introduce a new system (motivated by \cite{QuadraticOde21}), obtained by the coordinate transformation $y(t) \coloneqq x(t) + \int_{-\infty}^t h(s;\kappa) \dd s$, which leads to the ODE

\begin{center}    
\begin{tcolorbox}[width=0.9\linewidth, colback=gray!10, colframe=gray!40, boxrule=0.5pt, arc=3mm, left=2mm, right=2mm, top=0mm, bottom=3mm, halign=center, valign=center]

\begin{align}\tag{y}\label{eqq::y}
    \Dot{y}(t) = \lambda - \left( y(t) - \int_{-\infty}^t h(s;\kappa) \dd s\right)^2.
\end{align}

\end{tcolorbox}
\end{center}

Both \eqref{eqq::x} and \eqref{eqq::y} appear frequently throughout this paper. We also refer to \eqref{eqq::x} as the $x$-system and \eqref{eqq::y} as the $y$-system. We denote by $t\mapsto x(t,s,x_0)$ the maximal solution of \eqref{eqq::x} satisfying $x(s,s,x_0) = x_0$. Equivalently, the maximal solution of \eqref{eqq::y} satisfying $y(s,s,x_0)=x_0$ is denoted by $t\mapsto y(t,s,x_0)$. Note that, given a set of parameters $\{\lambda,\kappa,h\}$, the $x$-system is in Case A if and only if the $y$-system is in Case A (see \autoref{pro::xysamecase}). The same holds for Cases B and C. To simplify the notation, we define
\begin{align*}
    H_\kappa(t) = \int_{-\infty}^t h(s;\kappa) \dd s= \left\{\begin{array}{cc}
        \int_{-\infty}^t \alpha(s) \dd s &  t\leq 0 \\
        \int_{-\infty}^0 \alpha(s) \dd s + \int_0^t g(s\kappa) \dd s = \int_{-\infty}^0 \alpha(s) \dd s + \frac{1}{\kappa}\int_0^{t\kappa} g(s) \dd s & t>0,
    \end{array}\right.
\end{align*}

where the second equality is achieved by the time-rescaling $l=s\kappa$. Thus, \eqref{eqq::y} becomes \\
$\Dot{y}(t) = \lambda -(y(t)-H_\kappa(t))^2$. We further define
\begin{align}\label{eqq::def_delta_limits}
    \delta^- \coloneqq \lim_{t\to-\infty} H_\kappa(t) = 0 \qquad \delta_\kappa^+ \coloneqq \lim_{t\to\infty} H_\kappa(t).
\end{align}
Note that $\delta_-$ does not depend on $\kappa$ and $\delta_\kappa^+$ exists since $h(\cdot;\kappa)\in L^1$ for $\kappa > 0$. \\
Transformation to the $y$-system has three major reasons. First, for the $x$-system the moving equilibria (the equilibria of the time frozen system $\dot{x}(t) = \lambda - h(\Bar{t};\kappa)- x^2(t)$) for $\Bar{t}\in\R$ vanish whenever $h(t;\kappa)>\lambda$. However, the $y$-system always has moving equilibria for any choice of $h$. This enables approximation techniques that rely on characteristics of the time frozen system.
Second, $h$ is potentially of low regularity, while $H_\kappa$ is at least continuous and of bounded variation, as it is an integral of a $L^1$ function. We later prove that this leads to $H_\kappa(t)$ depending uniformly continuously on $\kappa$.
Third, the $y$-system is of the form $\Dot{y} = f(t, y-H_\kappa(t))$, and hence can be thought of as a transition equation between the limit equations
$$
    \Dot{y}_- = f(t, y_-\delta^-) \qquad \text{and} \qquad \Dot{y}_+ = f(t, y_+ - \delta_\kappa^+).
$$
This fits neatly in the framework considered in \cite[Chap.~4]{Longo_2024CriticalTransitionsForCaratheodory}.

\begin{remark}
    In the following, we study rate-induced tipping of the system \eqref{eqq::x}, i.e.~changes in qualitative behaviour of the attractor upon variation of $\kappa$. For the y-system \eqref{eqq::y}, the prefactor of $1/\kappa$ in front of the integral means that we actually analyse a combination of rate- and size-induced tipping (see \cite{Longo_2024CriticalTransitionsForCaratheodory} for definitions). 
    This balance is the main difference to most of the literature mentioned above and also one of the major difficulties when implementing numerical schemes to calculate critical rates such as in \cite{Kuehn2022, Tipping_thresholdsWieczorek_2023}. 
\end{remark}

\subsection{Basic notions of nonautonomous dynamics}\label{sec::nonautonomoustheory}
Let $f:\R^2 \to \R$ and consider a general nonautonomous scalar ODE
\begin{align}\label{eqq::general}
    \Dot{x} = f(t,x).
\end{align} 
\begin{conditions}\label{cond::1}
    We say that \autoref{cond::1} is satisfied for a general scalar ODE \eqref{eqq::general} if the following holds:
    \begin{enumerate}[label=(\roman*)]
        \item $f$ is Borel measurable.
        \item For all $j\in\N$, there exists $m_j \in \R$ such that, $l$-a.a., $\sup_{x\in[-j,j]} |f(t,x)| \leq m_j$.
        \item The map $x\mapsto f(t,x)$ is $C^1$ $l$-a.a.
        \item For all $j\in \N$, there exists $l_j\in\R$ such that, $l$-a.a.,
        $$
        \sup_{x_1,x_2\in [-j,j]} \frac{|f(t,x_2) - f(t,x_1)|}{|x_2-x_1|} \leq l_j \qquad \sup_{x_1,x_2\in [-j,j]} \frac{|f_x(t,x_2) - f_x(t,x_1)|}{|x_2-x_1|} \leq l_j.
        $$
    \end{enumerate}
\end{conditions}

\autoref{cond::1} ensures the existence of unique solutions (see \cite{odetheory}) and is also crucial for the rest of this section. Note that both the $x$-system and the $y$-system satisfy \autoref{cond::1} (see Appendix \autoref{app::prop::cond_allsystems}). For the rest of Section \ref{sec::nonautonomous_fold}, we always assume \autoref{cond::1} to be in place for any right hand side $f(t,x)$ we consider. 

\begin{definition}
We say that a bounded solution $b:\R\to\R$ of \eqref{eqq::general} is \textit{hyperbolic} if the corresponding variational equation $\Dot{z} = f_x(t,b(t))z$ admits an exponential dichotomy on $\R$, that is, there exists $K\geq1,$ $\beta>0$ such that one of the following two holds
\begin{align}
    & \exp \left(\int_s^t f_x(l,b(l)) \dd l \right) \leq K e^{-\beta (t-s)} && \qquad t\geq s \label{eqq::hyperbolicattracting}\\
    &\exp \left( \int_s^t f_x(l,b(l)) \dd l \right) \leq K e^{\beta (t-s)} && \qquad t\leq s. \label{eqq::hyperbolicrepelling}
\end{align}
In case of \eqref{eqq::hyperbolicattracting} we say that the hyperbolic solution $b$ is (locally) \textit{attractive} and in case of \eqref{eqq::hyperbolicrepelling} we say that the hyperbolic solution $b$ is (locally) \textit{repelling}. In both cases $(K,\beta)$ is called a dichotomy constant pair (\cite{Longo_2024CriticalTransitionsForCaratheodory, dichotomiesInStabilityTheory}). 
\end{definition}

The next result shows that hyperbolic solutions persist under small perturbations of $f$.

\begin{proposition}\label{pro::robustnesshyperbolic}
Let $\alpha,\Tilde{\alpha} \in \SET(\R)$ and $p\in C^2(\R)$. Consider the equations
\begin{align}
    & \dot{x}(t) = \alpha(t) + p(x) \label{eqq::prop::xhyperbolic}\\
    &\dot{z}(t) = \Tilde{\alpha}(t) + p(z) \label{eqq::prop::zhyperbolic}
\end{align}
and let $b^x$ be an attractive (repulsive) hyperbolic solution of \eqref{eqq::prop::xhyperbolic}. Then there exists $\delta^*>0$ such that for any $\delta \in (0,\delta^*)$, if $||\alpha - \Tilde{\alpha}||_{\infty;\R} < \delta $, then \eqref{eqq::prop::zhyperbolic} also has an attractive (repulsive) hyperbolic solution $b^z$.
\end{proposition}

The above holds in much more generality (see \cite[Prop.~3.3]{Longo_2024CriticalTransitionsForCaratheodory}). To see that it applies to our setting, note that our assumptions on $\alpha,\Tilde{\alpha},p$ trivially make the right hand sides satisfy \autoref{cond::1}. Further, 3 and 4 in \cite[Prop.~3.3]{Longo_2024CriticalTransitionsForCaratheodory} are satisfied, since the state derivatives of \eqref{eqq::prop::xhyperbolic} and \eqref{eqq::prop::zhyperbolic} coincide.
\\
\\
The following Proposition establishes the earlier mentioned connection between the $x $-system and the $y$-system.

\begin{proposition}\label{pro::xysamecase}
A solution $b$ of \eqref{eqq::x} is bounded if and only if $b + H_\kappa$ is bounded for \eqref{eqq::y}. Furthermore, a solution $b$ of \eqref{eqq::x} is hyperbolic attracting (repelling) if and only if $b + H_\kappa$ is hyperbolic attracting (repelling) for \eqref{eqq::y}.
\end{proposition}
\begin{proof}
    The first statement readily follows from $h\in L^1(\R)$ and therefore $|H_\kappa| \leq \frac{1}{\kappa} ||h||_{1;\R}$ and $y = x+ H_\kappa$. \\
    For the second statement, we start with the attracting case, the repelling one follows analogously. Assume $b$ is a hyperbolic attracting solution for \eqref{eqq::x}, i.e.~there exists a dichotomy constant pair $(K,\beta)$ such that
    \begin{align*}
        \exp\left(\int_s^t -2 b(l) \dd l \right) \leq Ke^{-\beta(t-s)} \qquad \text{for } t \geq s.
    \end{align*}
    Further, let $c \coloneqq b + H_\kappa$ and $f(t,x) \coloneqq \lambda -(x - H_\kappa(t))^2$. Then
    \begin{align*}
        & \exp \left(\int_s^t f_x(l,c) \dd l  \right)
        = \exp \left(\int_s^t  -2\left(c(l) - H_\kappa(l) \right) \dd l  \right) = \exp \left(\int_s^t  -2b(l) \dd l \right) 
        \leq K e^{\beta (t-s)}
    \end{align*}
    for $t\geq s$. Hence $c$ is hyperbolic attracting for \eqref{eqq::y} with the same dichotomy constant pair $(K,\beta)$. This shows ``$\Longrightarrow$", for ``$\Longleftarrow$" follows the same way by reading the equations from bottom to top. 
\end{proof}
Note that \autoref{pro::xysamecase} enables us to identify Cases A,B,C for \eqref{eqq::x} with the same case for \eqref{eqq::y}. In particular for a set of parameters $\{\lambda, \kappa, h\}$, \eqref{eqq::x} and \eqref{eqq::y} are always in the same Case. Nevertheless, both equations are kept throughout this paper, as some results are easier to obtain for \eqref{eqq::x}, while for others the formulation \eqref{eqq::y} is more convenient.

\subsection{Nonautonomous saddle node bifurcation}\label{sec::saddlenode}
The following \autoref{thm::casesx} and \autoref{thm::casesy} describe the dynamical behaviour of \eqref{eqq::x} and \eqref{eqq::y}.  
They are a summary of the results obtained in \cite{Longo_2024CriticalTransitionsForCaratheodory, QuadraticOde21, PiecewiseUniformlyContinuous24} and are put in our context here.

\begin{theorem}\label{thm::casesx}
Consider the $x$-system \eqref{eqq::x}, then the following holds:
\begin{enumerate}[label=(\roman*)]
    \item There exists a maximal solution $\attr^x$ which is defined on a negative open halfline $\mathcal{R}^-\coloneqq (-\infty,t_f)$ or $\R$ and is bounded from above, such that $x(t,s,x_0)$ remains bounded as $t\to-\infty$ if and only if $x_0 \leq \attr^x(s)$.
    \item There exists a maximal solution $\rep^x$ which is defined on a positive open halfline $\mathcal{R}^+\coloneqq (t_0,\infty)$ or $\R$ and is bounded from below, such that $x(t,s,x_0)$ remains bounded as $t\to\infty$ if and only if $x_0 \geq \rep^x(s)$.
    \item $\attr^x$ is locally pullback attractive and $\rep^x$ is locally pullback repulsive.
    \item Any globally defined solution is bounded.
    \item If there exist bounded solutions, then $\attr^x$ and $\rep^x$ are globally defined, and bounded solutions are exactly those $x(t,s,x_0)$ where $\rep^x(s) \leq x_0 \leq \attr^x(s)$.
\end{enumerate}
Further, one of the following holds
\begin{itemize}
    \item Case A: The two solutions $\attr^x$ and $\rep^x$ are hyperbolic with $\attr^x$ attractive and $\rep^x$ repulsive. Moreover,
    \begin{align*}
       & \lim_{t\to-\infty} \attr^x(t) = \sqrt{\lambda} \qquad \qquad && \lim_{t\to\infty} \attr^x(t) = \sqrt{\lambda} \\
       & \lim_{t\to-\infty} \rep^x(t) = -\sqrt{\lambda} \qquad \qquad && \lim_{t\to\infty} \rep^x(t) = -\sqrt{\lambda} 
    \end{align*}
    and $\lim_{t\to-\infty} x(t,s,x_0) = - \sqrt{\lambda}$ for $x_0 < \attr^x(s)$ as well as $\lim_{t\to\infty} x(t,s,x_0) = \sqrt{\lambda}$ for $x_0 > \rep^x(s)$. Finally, $\attr^x,\rep^x$ are the only hyperbolic solutions.
    \item Case B: There exist no hyperbolic solutions, $\attr^x = \rep^x$ is the only bounded solution, and
    \begin{align*}
        \lim_{t\to-\infty} \attr^x(t) = \sqrt{\lambda} \qquad \qquad \lim_{t\to\infty} \attr^x(t) = -\sqrt{\lambda}.
    \end{align*}
    \item Case C: There exist no globally defined solutions, and
    \begin{align*}
        \lim_{t\uparrow t_f} \attr^x (t) = - \infty \qquad \qquad \lim_{t\downarrow t_0} \rep^x (t) =  \infty .
    \end{align*}
\end{itemize}
\end{theorem}

\begin{theorem}\label{thm::casesy}
Consider the $y$-system \eqref{eqq::y}, then the following holds:
\begin{enumerate}[label=(\roman*)]
    \item There exists a maximal solution $\attr^y$, which is defined on a negative open halfline $\mathcal{R}^-\coloneqq (-\infty,t_f)$ or $\R$ and is bounded from above, such that $y(t,s,x_0)$ remains bounded as $t\to-\infty$ if and only if $x_0 \leq \attr^y(s)$.
    \item There exists a maximal solution $\rep^y$, which is defined on a positive open halfline $\mathcal{R}^+\coloneqq (t_0,\infty)$ or $\R$ and is bounded from below, such that $y(t,s,x_0)$ remains bounded as $t\to\infty$ if and only if $x_0 \geq \rep^y(s)$.
    \item $\attr^y$ is locally pullback attractive and $\rep^y$ is locally pullback repulsive.
    \item Any globally defined solution is bounded.
    \item If there exist bounded solutions, then $\attr^y$ and $\rep^y$ are globally defined and the bounded solutions are exactly those $y(t,s,x_0)$ where $\rep^y(s) \leq x_0 \leq \attr^y(s)$.
\end{enumerate}
Further, one of the following holds
\begin{itemize}
    \item Case A: The two solutions $\attr^y$ and $\rep^y$ are hyperbolic with $\attr^y$ attractive and $\rep^y$ repulsive. Moreover,
    \begin{align*}
       & \lim_{t\to-\infty} \attr^y(t) = \sqrt{\lambda} + \delta^- \qquad \qquad && \lim_{t\to\infty} \attr^y(t) = \sqrt{\lambda} + \delta_\kappa^+ \\
       & \lim_{t\to-\infty} \rep^y(t) = -\sqrt{\lambda} + \delta^- \qquad \qquad && \lim_{t\to\infty} \rep^y(t) = -\sqrt{\lambda} + \delta_\kappa^+
    \end{align*}
    and $\lim_{t\to-\infty} y(t,s,x_0) = - \sqrt{\lambda} + \delta^-$ for $x_0 < \attr^y(s)$ as well as $\lim_{t\to\infty} y(t,s,x_0) = \sqrt{\lambda} + \delta_\kappa^+$ for $x_0 > \rep^y(s)$, where $\delta^-,\delta^+_\kappa$ are from \eqref{eqq::def_delta_limits}. Finally, $\attr^y,\rep^y$ are the only hyperbolic solutions.
    \item Case B: There exist no hyperbolic solutions, $\attr^y = \rep^y$ is the only bounded solution and
    \begin{align*}
        \lim_{t\to-\infty} \attr^y(t) = \sqrt{\lambda} + \delta^- \qquad \qquad \lim_{t\to\infty} \attr^y(t) = -\sqrt{\lambda} + \delta_\kappa^+.
    \end{align*}
    \item Case C: There exist no globally defined solutions, and
    \begin{align*}
        \lim_{t\uparrow t_f} \attr^y (t) = - \infty \qquad \qquad \lim_{t\downarrow t_0} \rep^y (t) =  \infty .
    \end{align*}
\end{itemize}
\end{theorem}

\begin{proof}[Proof of \autoref{thm::casesx} and \autoref{thm::casesy}]
   After noticing that both \eqref{eqq::x} and \eqref{eqq::y} are coercive and concave (see Appendix \autoref{app::prop::cond_allsystems} for definition and proof) statements (i)-(iv) follow from \cite[Thm.~3.1,~Thm.~3.4]{Longo_2024CriticalTransitionsForCaratheodory}, since their conditions f1-f5 are covered by our \autoref{cond::1}.
    \\\\
   For the statements about the different cases; note that the $y$-system \eqref{eqq::y} is of the form $\dot{y} = f(t, y - H_\kappa(t))$, where $H_\kappa \in L^\infty(\R)$ (since $h \in L^1$) and moreover, $\lim_{t\to\pm\infty} H_\kappa(t)$ exist. Hence, the $y$-system fulfils all the assumptions of 
   \cite[Thm.~4.4]{Longo_2024CriticalTransitionsForCaratheodory}, which proves the result. The corresponding results for the $x$-system \eqref{eqq::x} then follow readily by \autoref{pro::xysamecase} and the fact that $x(t) = y(t) + H_\kappa (t)$.
\end{proof}

\begin{remark}
    Up to this point, we identified $\attr^{x,y}$ ($\rep^{x,y}$) as pullback attractors (repellers). Note that the nonautonomous nature of \eqref{eqq::x} and \eqref{eqq::y} demands a distinction between pullback and forward attractors (repellers). \autoref{thm::casesx} and \autoref{thm::casesy} show that $\attr^{x,y}$ and $\attr^{x,y}$ are forward attractors/ repellers, i.e.~they form a classical attractor-repeller pair of a hyperbolic solution (\cite{Longo_2024CriticalTransitionsForCaratheodory}), if and only if they are in Case A. However, in any case, they are pullback attractors/repellers. Since the properties of pullback/forward attractors/repellers, apart from what has already been shown in \autoref{thm::casesx} and \autoref{thm::casesy}, are not relevant for the rest of this paper, we sometimes simply call $\attr^{x,y}$ attractor and $\rep^{x,y}$ repeller going forward.
\end{remark}

\subsection{Introduction of a rate function \texorpdfstring{$\Kappa$}{Kappa}}\label{sec::ratefunction}

In this section, we analyse how, for a fixed $\lambda>0$, \eqref{eqq::x} and \eqref{eqq::y} transition between Cases A,B and C upon variation of $\kappa$. Throughout \autoref{thm::casesopen} to \autoref{cor::transversalTipping} we establish that (under certain weak assumptions on $\alpha$ and $g$) there is a unique rate $\Bar{\kappa}$ such that: 
\begin{itemize}
        \item \eqref{eqq::x}$_{\kappa}$ is in Case A for $\kappa > \Bar{\kappa}$;
        \item \eqref{eqq::x}$_{\kappa}$ is in Case B for $\kappa = \Bar{\kappa}$;
        \item \eqref{eqq::x}$_{\kappa}$ is in Case C for $\kappa < \Bar{\kappa}$.
\end{itemize}
In particular, this means that the boundedness of $\attr^x_\kappa$ can be enforced by increasing $\kappa$.
A key result in this section is \autoref{thm::transversaltipping} which establishes that for any system of the form \eqref{eqq::x} there can be maximally one transition from Case A to C upon varying $\kappa$. \\
The strategy in this chapter is as follows: In \autoref{thm::casesopen}, we show that any transition from Case A to C must go through Case B, then in \autoref{thm::largekappa}, we establish that, upon making $\kappa$, large enough \eqref{eqq::x}$_\kappa$ is always in Case A. Hence, we can ``save" the attractor by increasing the rate. As a next step, we establish in \autoref{thm::transversaltipping} that there is at most one transition between Case A and C. These results together then culminate in our main statement above (\autoref{cor::transversalTipping}) and the definition of the rate function $\Kappa$, which assigns every set of parameters $\{\lambda, \alpha, g\}$ the $\Bar{\kappa}$ such that \eqref{eqq::x}$_{\Bar{\kappa}}$ is in Case B (see \autoref{def::Kappa}). This structure is also visualised in \autoref{fig::flow_diagram}.

\begin{figure}
    \centering
    \begin{tikzpicture}[
        node distance=1.3cm and 1.6cm, 
        box/.style={rectangle, rounded corners, draw, thick, align=center, text width=4.2cm, minimum height=1.4cm, inner sep=8pt},
        theorem/.style={
            box,
            fill=blue!8
        },
        corollary/.style={
            box,
            fill=green!10
        },
        definition/.style={
            box,
            fill=orange!12
        },
        arrow/.style={
            -{Latex[length=3mm]},
            thick
        }
    ]
    
    \node[theorem] (T38) {
        \underline{\autoref{thm::casesopen}}\\[2mm]
        Any transition from Case A to Case C must go through Case B.
    };
    
    \node[theorem, right=of T38] (T310) {
        \underline{\autoref{thm::largekappa}}\\[2mm]
        System in Case A for large $\kappa$.
    };
    
    \node[theorem, right=of T310] (T311) {
        \underline{\autoref{thm::transversaltipping}}\\[2mm]
        We have at most one transition.
    };

    \node[corollary, below=of T310] (C313) {
        \underline{\autoref{cor::transversalTipping}}\\[2mm]
        If $g$ is such that there exists a tipping rate, then there exists exactly one rate-induced transition.
    };
    
    \node[definition, below=of C313] (D315) {
        \underline{\autoref{def::Kappa}}\\[2mm]
        Definition of rate function $\Kappa$.
    };

    \draw[arrow] (T38.south east) -- (C313.north west);
    \draw[arrow] (T310.south) -- (C313.north);
    \draw[arrow] (T311.south west) -- (C313.north east);
    \draw[arrow] (C313.south) -- (D315.north);

\end{tikzpicture}
    \caption{Flow chart of the structure of this section.}
    \label{fig::flow_diagram}
\end{figure}
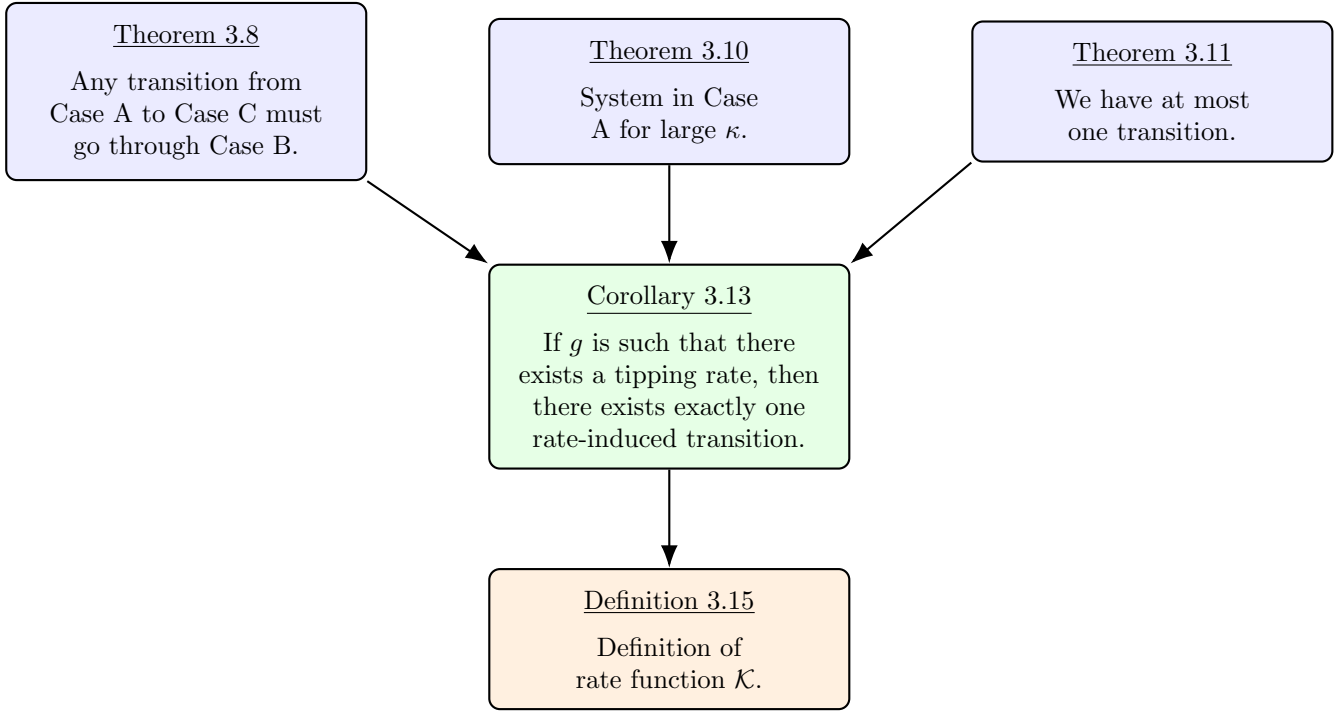

\begin{theorem}\label{thm::casesopen}
    For the $y$-system \eqref{eqq::y} we have:
    \begin{enumerate}[label=(\roman*)]
        \item The set $\{\kappa>0: \eqref{eqq::y}_\kappa \text{ is in Case A} \}$ is open;
        \item The set $\{\kappa>0: \eqref{eqq::y}_\kappa \text{ is in Case C} \}$ is open;
        \item If there exist $0<\kappa_1< \kappa_2 $ such that \eqref{eqq::y}$_{\kappa_1}$ is in Case C and \eqref{eqq::y}$_{\kappa_2}$ is in Case A, then, there exists $\kappa \in (\kappa_1,\kappa_2)$ such that \eqref{eqq::y}$_{\kappa}$ is in Case B. In particular, any transition from Case A to Case C must go through Case B.
    \end{enumerate}
 By \autoref{pro::xysamecase} the same holds for the $x$-system \eqref{eqq::x}.
\end{theorem}

The proof of \autoref{thm::casesopen} requires the following Lemma, the proof of which has been moved to the Appendix.
\begin{lemma}\label{lem::propH}
The following holds for $H_\kappa$:
\begin{enumerate}[label=(\roman*)]
    \item Let $\kappa>0$, then the maps $t\mapsto H_\kappa(t)$ and $t\mapsto H_\kappa(t)^2$ are bounded and uniformly continuous.
    \item Let $0<a<b$ and $(\kappa_n)_{n\in\N}$ be a sequence in $[a,b]$ with $\kappa_n \to \kappa>0$. Then there exists a subsequence $(\kappa_{n_k})_{k\in\N}$ such that
    \begin{align*}
        &H_{\kappa_{n_k}} \longrightarrow H_\kappa \quad \text{uniformly, and} \\
        &H_{\kappa_{n_k}}^2 \longrightarrow H_\kappa^2 \quad \text{uniformly.}
    \end{align*}
\end{enumerate}
\end{lemma}

\begin{proof}[Proof of \autoref{thm::casesopen}]
There exists $\lambda^* = \lambda^*(2H_\kappa, -H_\kappa^2): C_{bu}(\R) \times C_{bu}(\R) \to \R$ that is continuous in the $L^\infty$ topology and is such that \eqref{eqq::y} is in Case A if and only if if $\lambda > \lambda^* $. Further, \eqref{eqq::y} is in Case C if and only if $\lambda < \lambda^* $, by Appendix \autoref{thm::lambdastar}. \\
We define
\begin{align}\label{eqq::def::lambdaB}
    \lambdaB : \R^+ \to \R,\  \lambdaB(\kappa) =  \lambda^*(2H_\kappa, -H_\kappa^2)
\end{align}
The rest of the argument is a consequence of the fact that the preimage of an open set under a continuous function is open. \\

(i) 
By Appendix \autoref{thm::lambdastar}
\begin{align*}
    \{\kappa>0: \eqref{eqq::y}_\kappa \text{ is in Case A} \} = \lambdaB^{-1}((-\infty,\lambda)) ,
\end{align*}
where $\lambdaB^{-1}(A)$ denotes the preimage of $A$ under $\lambdaB$. We show that $\lambdaB^{-1}((-\infty,\lambda))$ is open, by showing that $V \coloneqq \R \backslash \lambdaB^{-1}((-\infty,\lambda))$ is closed. Denote by $W$ the closed set $W \coloneqq [\lambda, \infty) = \R \backslash (-\infty,\lambda)$. Let $(\kappa_n)_{n\in\N}$ be a sequence in $V$, with $\kappa_n \to \kappa$. By \autoref{lem::propH} and the fact that $\lambda^*$ is continuous in the $L^\infty$ topology, there exists a subsequence $\kappa_{n_k}$ such that $\lambdaB(\kappa_{n_k}) \to \lambdaB(\kappa)$. Since $W$ is closed, $\lambdaB(\kappa) \in W$ and therefore $\kappa \in V = \lambdaB^{-1}(W)$. \\
The proof of (ii) follows the same argumentation, since
\begin{align*}
        \{\kappa>0: \eqref{eqq::y}_\kappa \text{ is in Case C} \} = \lambdaB^{-1}((\lambda,\infty)) .
\end{align*}
(iii) Assume that there exist $0<\kappa_1<\kappa_2 $ such that \eqref{eqq::y}$_{\kappa_1}$ is in Case C and in \eqref{eqq::y}$_{\kappa_2}$ Case A. Let $X \coloneqq [\kappa_1,\kappa_2]$ equipped with the subspace topology. Note that $X$ is connected and hence $\emptyset,X$ are the only sets that are simultaneously open and closed in $X$. Further,
\begin{align*}
    & \mathscr{A} \coloneqq X \cap \{\kappa>0: \eqref{eqq::y}_\kappa \text{ is in Case A} \} \\
    & \mathscr{C} \coloneqq X \cap \{\kappa>0: \eqref{eqq::y}_\kappa \text{ is in Case C} \}
\end{align*}
are both open and non-empty $X$. The former follows from the definition of the subspace topology, while the latter follows from $\kappa_2 \in \mathscr{A}$ and $\kappa_1 \in \mathscr{C}$. If there was no $\kappa \in [\kappa_1,\kappa_2]$ such that \eqref{eqq::y}$_\kappa$ is in Case B, then by \autoref{thm::casesy}
\begin{align*}
   X = [\kappa_1,\kappa_2] = \mathscr{A} \cup \mathscr{C}.
\end{align*}
In particular, $\mathscr{A} = X\backslash\mathscr{C}$ must be closed. Hence, $\mathscr{A}$ is an open and closed set in $X$. However, since $\mathscr{A}$ and $\mathscr{C}$ are both non-empty, $\mathscr{A} \notin \{\emptyset,X\}$, a contradiction. It follows that $\mathscr{A} \cup \mathscr{C} \subsetneq X$ and therefore there exists a $\kappa\in[\kappa_1,\kappa_2]$ such that \eqref{eqq::y}$_\kappa$ is in Case B.
\end{proof}

So far, we considered $\attr^x$ (resp. $\attr^y$) as the attractor of \eqref{eqq::x} (resp. \eqref{eqq::y}). This has been done to make use of the results obtained in \cite{Longo_2024CriticalTransitionsForCaratheodory, QuadraticOde21, PiecewiseUniformlyContinuous24}. However, in the optimal control problem \eqref{prob::OCP} one considers a trajectory with $x(0)=\mu^x$. As a consequence, we aim to also formulate some result of this section in terms of an initial value problem. To that end, note that $\attr^x$ and $\attr^y$ are the solutions to
\begin{align}
    \label{eqq::initalvaluex}&\dot{x}_\kappa(t) = \lambda -g(t\kappa) - x_\kappa^2 \qquad &x_\kappa(0)=\attr^x(0) \\
    \label{eqq::initalvaluey}&\dot{y}_\kappa(t) = \lambda - \left(y_\kappa(t) - \int_{-\infty}^0 \alpha(s) \dd s - \int_0^t g(s\kappa) \dd s \right)^2 \qquad &y_\kappa(0) = \attr^y(0)
\end{align}
for $t\geq 0$. It is easy to see that the uniqueness of solutions implies that $\attr^x$ and $\attr^y$ are actually the solutions to \eqref{eqq::initalvaluex} and \eqref{eqq::initalvaluey}, respectively, provided that $\attr^x(0)$ (resp. $\attr^y(0)$) exist. This will be an ongoing assumption throughout the rest of this paper. Note that this assumption is well justified since if $\attr^x(0)$ does not exist, this means that \eqref{eqq::x} diverges to $-\infty$ before time $0$. Hence, the event one wants to avoid by the control occurs before any control can be applied. This would be an ill-posed control problem. In the following, we make a slightly stronger assumption on $\alpha$, namely that it is such that $\attr^x(0) > -\sqrt{\lambda}$. This, in particular, means that completely undoing the forcing for $t\geq 0$ (i.e.~$u=\alpha$) results in a bounded trajectory.

\begin{theorem}\label{thm::largekappa}
    Let $\attr^x(0) > -\sqrt{\lambda}$, then, for any $g\in\SET(\R^+)$, there exists $\kappa_2$ such that for any $\kappa > \kappa_2$,
    \begin{enumerate}[label=(\roman*)]
        \item \eqref{eqq::x}$_\kappa$ and \eqref{eqq::y}$_\kappa$ are in Case A.
        \item If in addition $g\geq 0$ and $\attr^x(0) \in (-\sqrt{\lambda}, \sqrt{\lambda})$, $\kappa_2$ can be chosen as $\kappa_2 = \frac{||g||_{1;[0,\infty)}}{\sqrt{\lambda} + \attr^x(0)}$.
        \item Further, \eqref{eqq::x}$_{\kappa_2}$ and \eqref{eqq::y}$_{\kappa_2}$ are in Case A or B.
    \end{enumerate} 
\end{theorem}

The intuition behind the proof is loosely based on forward basin stability (see \cite{Ashwin2017}) and is best explained for $\attr^x(0) \in (-\sqrt{\lambda}, \sqrt{\lambda})$ and $g\geq0$: The $y$-system \eqref{eqq::y} describes a nonautonomous ODE, where the equilibria of the time frozen system ($\pm \sqrt{\lambda} + \int_{-\infty}^{\Bar{t}} h(t;\kappa) \dd t$) at any time $\Bar{t}$, the moving equilibria, are moving with time. At time $t=0$, they arrive at $\pm \sqrt{\lambda} + \int_{-\infty}^{0} \alpha(t) \dd t$ and the assumption on $\attr^x(0)$ ensures that $\attr^y(0)$ is in between the two moving equilibria. Furthermore, between the two moving equilibria, $\attr^y$ has positive derivative. Note that $\int_0^t g(s\kappa) \dd s = \frac{1}{\kappa} \int_0^{t\kappa} g(l) \dd l$. Hence, upon varying $\kappa$ one can arbitrarily change how far the moving equilibria move after $t=0$. In particular, if one choses $\kappa>0$ large enough such that the unstable equilibrium moves less than the distance between $\attr^y(0)$ and the equilibrium at $t=0$ ($-\sqrt{\lambda} + \int_{-\infty}^0 \alpha(s) \dd s$), then the positive derivative between the moving equilibria ensures that $\attr^y$ stays in the basin of attraction of the stable moving equilibrium for all $t\geq0$. Therefore, $\lim_{t\to\infty}\attr^y = \sqrt{\lambda} +\int_\R h(t;\kappa) \dd t$, i.e.~is in Case A. This is visually illustrated in Figure \ref{fig::moving_equilibria}. \\

The formal proof has been moved to the Appendix.
\newline 
\begin{figure}[h!]
    \centering
    \includegraphics[width=0.99\linewidth]{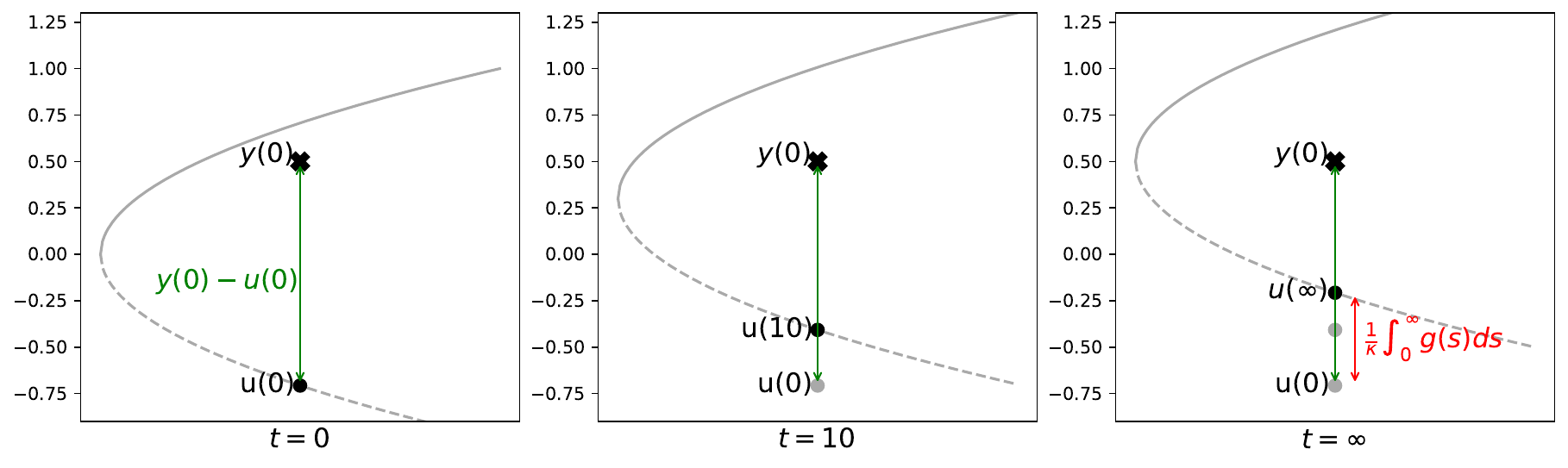}
    \caption{Moving equilibria for the \eqref{eqq::y}-system. $u(t)$ denotes the position of the unstable moving equilibrium at time $t$, $y(0)=\attr^y(0)$. The solid and dashed curved lines correspond to the bifurcation diagram of the time frozen system. We see that as long as $y(0)-u(0) \geq u(\infty) - u(0) = \frac{1}{\kappa} \int_0^\infty g(s) \dd s$, no tipping can occur. Solving for $\kappa$ gives the bound of \autoref{thm::largekappa} (ii).}
    \label{fig::moving_equilibria}
\end{figure}

The following Theorem shows that in the nonautonomous saddle node normal form \eqref{eqq::x}, there is at most one rate-induced transition between Cases A and C and further this transition is transversal, i.e. the transition interval, where \eqref{eqq::x}$_\kappa$ (resp. \eqref{eqq::y}$_\kappa$) is in Case B is a singleton.

\begin{theorem}\label{thm::transversaltipping}
    Let $\alpha \in \SET(\R)$ be such that $\attr^x(0) > 0$ and let $g \in \SET(\R^+)$.
    \begin{enumerate}[label=(\roman*)]
        \item Assume that $\kappa_1$ is such that \eqref{eqq::x}$_{\kappa_1}$ and \eqref{eqq::y}$_{\kappa_1}$ are in Case C. Then \eqref{eqq::x} and \eqref{eqq::y} are in Case C for any $\kappa<\kappa_1$.
        \item Assume there is $\Bar{\kappa}$ such that \eqref{eqq::x}$_{\Bar{\kappa}}$ and \eqref{eqq::y}$_{\Bar{\kappa}}$ are in Case B, then they are not in Case B for any $\kappa < \Bar{\kappa}$.
    \end{enumerate}
\end{theorem}

\begin{proof}
    We argue via the \eqref{eqq::x} system. The results for \eqref{eqq::y} follow from \autoref{pro::xysamecase}. 
    Fix $\alpha\in \SET(\R), \ g \in \SET(\R^+)$ with $\alpha$ such that $\attr^x(0) > 0$ and let $\kappa<\kappa_1$. Note, for $t\geq 0$, $\attr_\kappa^x$ is the solution $x_\kappa$ to 
    \begin{align*}
        \dot{x}_\kappa(t) = \lambda - g(t\kappa) - x_\kappa^2(t) \qquad x_\kappa(0) = \attr_\kappa^x(0) \eqqcolon \mu^x.
    \end{align*}
    (i) By \autoref{thm::casesx} Case C means that $x_\kappa$ has a finite time blowup to $-\infty$. We define $\phi$ by $x_\kappa = \phi_\kappa' / \phi_\kappa$. Then $\phi_\kappa$ follows the second order linear ordinary differential equation
    \begin{align*}
        \phi_\kappa''(t) = (\lambda - g(t\kappa)) \phi_\kappa(t) \qquad\qquad \phi_\kappa(0) \mu^x = \phi_\kappa'(0).
    \end{align*}
    for $\phi(0) \neq 0$. Note, there is nothing mathematically wrong with choosing $\phi(0) =0$. However, then $\phi(t)=\phi'(t) = \phi''(t)=0$ for all $t\geq0$ and we lose any interpretability of $x=\phi'/\phi =0/0$. After time-rescaling $\varphi_\kappa (t) = \phi_\kappa(t/\kappa)$, we get
    \begin{align*}
        \varphi_\kappa''(t) = \frac{\lambda - g(t)}{\kappa^2} \varphi_\kappa(t) \qquad \qquad \varphi_\kappa(0) \mu^x \kappa = \varphi_\kappa'(0).
    \end{align*}
    Since $g\in L^\infty(\R^+)$, and therefore $\phi_\kappa'$ (resp. $\varphi_\kappa'$) grow at most exponentially, the finite time explosions of $x_\kappa$ correspond to the roots of $\phi_\kappa$ (resp. $\varphi_\kappa$). We show that if $\varphi_{\kappa_1}$ has a root for $\kappa_1>0$ (i.e.~$x_{\kappa_1}$ is in Case C), then $\varphi_\kappa$ also has a root for any $0<\kappa<\kappa_1$ (i.e.~$x_\kappa$ is in Case C). To do so, we apply the Pr\"ufer transformation (see, e.g.~\cite{pruefer, Pruefer1926})
    \begin{align*}
        \varphi_\kappa(t) = \rho_\kappa(t) \sin(\theta_\kappa(t)) \qquad \varphi_\kappa'(t) = \frac{1}{\kappa^2} \rho_\kappa(t) \cos(\theta_\kappa(t)),
    \end{align*}
    where it can be shown that without loss of generality $\rho_\kappa > 0$ (\cite{Pruefer1926}). In particular, this means that the roots of $\varphi_\kappa$ (and therefore the finite time explosions of $x_\kappa$) correspond to the times $t\geq 0$, where $\theta_\kappa(t) = k \pi$ for $k\in\mathbb{Z}$.  Moreover, $\theta$ follows the ODE (\cite{pruefer})
    \begin{align}\label{eqq::proof::theta}
        \theta_\kappa'(t) = \frac{1}{\kappa^2} \cos^2(\theta_\kappa(t)) + (g(t) - \lambda) \sin^2(\theta_\kappa(t)).
    \end{align}

    Further, the initial condition transforms as follows: 
    \begin{align*}
       \varphi_\kappa(0) \mu^x \kappa &= \varphi_\kappa'(0) 
       \Longleftrightarrow \mu^x \kappa \rho_\kappa(0)\sin(\theta_\kappa(0)) = \frac{1}{\kappa^2} \rho_\kappa(0)\cos(\theta_\kappa(0)) \\
       &\Longleftrightarrow \theta_\kappa(0) = \arctan\left(\frac{1}{\kappa^3 \mu^x}\right) + l \pi \qquad l \in \mathbb{Z} .
    \end{align*}
    It is easy to see that \eqref{eqq::proof::theta} is $\pi$-periodic, hence we can, without loss of generality, take $l = 0$ and obtain
    \begin{align}\label{eqq::proof::thetafinal}
        \theta_\kappa'(t) = \frac{1}{\kappa^2} \cos^2(\theta_\kappa(t)) + (g(t) - \lambda) \sin^2(\theta_\kappa(t)) \qquad \theta_\kappa(0) = \arctan\left(\frac{1}{\kappa^3 \mu^x}\right).
    \end{align}
    Note that $\theta_\kappa(0) \in (0,\pi/2)$, since $\mu^x,\kappa>0$. \\
    We show two things: First, the solution $\theta_\kappa$ to \eqref{eqq::proof::thetafinal} is decreasing in $\kappa$ and second that $\theta_\kappa(t) > 0$ for all $\kappa>0$ and all $t\geq0$. Together, this shows that if $x_{\kappa_1}$ is in Case C (i.e.~$\theta_{\kappa_1}$ crosses $k\pi$ for some $k\in\mathbb{Z}$), then the continuity of $\theta_{\kappa_1}$ implies that there exists $\Bar{t} >0$ such that $\theta_{\kappa_1}(\Bar{t}) = \pi$. However, since \eqref{eqq::proof::thetafinal} is decreasing in $\kappa$, this means that $\theta_{\kappa}(\Bar{t}) \geq \pi$ for all $0<\kappa<\kappa_1 $. In particular, by the continuity of $\theta_\kappa$ this means that for all $0<\kappa<\kappa_1$ there also exists $t$ s.t.~$\theta_\kappa (t) = \pi$, i.e.~$x_\kappa$ is in Case C.
    \\\\
    For the first part, note that since $\theta_\kappa'$ and $\theta_\kappa(0)$ are decreasing in $\kappa$, we find that $\theta_\kappa < \theta_{\Bar{\kappa}}$ for $\Bar{\kappa} < \kappa$, by the comparison principle \autoref{thm::comparison}. We now show that $\theta_\kappa (t) > 0$ for all $t \geq 0$. For that, define
    \begin{align*}
        m \coloneqq \left\{\begin{array}{cc}
           \inf_{t\geq 0} g(t)  & \inf_{t\geq 0} g < 0 \\
            0 & \inf_{t\geq 0} g \geq 0
        \end{array}\right. 
        \qquad\qquad \Tilde{g} = g - m \geq 0
        \qquad\qquad \Tilde{\lambda} = \lambda - m > 0.
    \end{align*}
    Then \eqref{eqq::proof::thetafinal} becomes 
    \begin{align}\label{eqq::proof::thetatilde}
        \theta_\kappa'(t) = \frac{1}{\kappa^2} \cos^2(\theta_\kappa(t)) + (\Tilde{g}(t) - \Tilde{\lambda}) \sin^2(\theta_\kappa(t)) \qquad \theta_\kappa(0) = \arctan\left(\frac{1}{\kappa^3 \mu^x}\right).
    \end{align}
    Moreover, define 
    \begin{align}\label{eqq::proof::eta}
        \eta_\kappa'(t) = \frac{1}{\kappa^2} \cos^2(\eta_\kappa(t)) - \Tilde{\lambda} \sin^2(\eta_\kappa(t)) \qquad \eta_\kappa(0) = \arctan\left(\frac{1}{\kappa^3 \mu^x}\right).
    \end{align}
    Then $\eta_\kappa$ is a scalar autonomous ODE and further $\eta_\kappa (t) \leq \theta_\kappa(t)$ for all $t\geq 0$ by the comparison principle \autoref{thm::comparison} and \autoref{app::prop::cond_allsystems}, since $\Tilde{g}\geq 0$. Further, $\eta_\kappa$ has two stationary points in $(-\pi/2,\pi/2)$ given by
    \begin{align*}
        \eta_\kappa^\pm = \pm \arctan \left( \frac{1}{\kappa\sqrt{\Tilde{\lambda}}}\right).
    \end{align*}
    Moreover, the positive one is attracting as 
    \begin{align*}
        \partial_{\eta_\kappa} \left(\eta_\kappa' \right)\bigg|_{\eta_\kappa^+} &= \left(-2\frac{1}{\kappa^2} \cos(\eta_\kappa) \sin(\eta_\kappa) - 2\Tilde{\lambda} \cos(\eta_\kappa)\sin(\eta_\kappa)  \right)\bigg|_{\eta_\kappa^+} = -2\left(\frac{1}{\kappa^2} + \Tilde{\lambda}\right) \cos^2(\eta_\kappa^+) \tan(\eta_\kappa^+) \\
        &= -2\left(\frac{1}{\kappa^2} + \Tilde{\lambda}\right)\frac{1}{\kappa \sqrt{\Tilde{\lambda}}} \cos^2(\eta_\kappa^+) < 0.
    \end{align*}
    Since $\eta_\kappa(0) > 0$, it follows that $\eta_\kappa(t) >0$ for all $t\geq 0$ and therefore
    \begin{align*}
        0< \eta_\kappa(t) \leq \theta_\kappa(t) \qquad \forall \  t \geq 0.
    \end{align*}

    (ii) By \autoref{thm::casesx}, Case B means that $\lim_{t\to\infty} x_{\Bar{\kappa}} = - \sqrt{\lambda}$. By the same steps as in (i) it follows that
    \begin{align*}
        - \sqrt{\lambda} &= \lim_{t\to\infty} x_{\Bar{\kappa}} 
        = \lim_{t\to\infty} \frac{\phi_{\Bar{\kappa}}'(t)}{\phi_{\Bar{\kappa}}(t)} 
        = \lim_{t\to\infty} \frac{\phi_{\Bar{\kappa}}'(t/\Bar{\kappa})}{\phi_{\Bar{\kappa}}(t/\Bar{\kappa})} 
        = \lim_{t\to\infty} \frac{\Bar{\kappa}}{\Bar{\kappa}} \frac{\phi_{\Bar{\kappa}}'(t/\Bar{\kappa})}{\phi_{\Bar{\kappa}}(t/\Bar{\kappa})} 
        = \lim_{t\to\infty} \Bar{\kappa} \frac{\varphi_{\Bar{\kappa}}'(t)}{\varphi_{\Bar{\kappa}}(t)} \\
        &= \lim_{t\to\infty} \frac{\rho_{\Bar{\kappa}}(t) \cos(\theta_{\Bar{\kappa}})}{\Bar{\kappa} \rho_{\Bar{\kappa}}(t) \sin(\theta_{\Bar{\kappa}})}
        = \lim_{t\to\infty} \frac{\cos(\theta_{\Bar{\kappa}}(t)) }{\Bar{\kappa} \sin(\theta_{\Bar{\kappa}}(t))}.
    \end{align*}
    In particular, this means that 
    \begin{align*}
        \lim_{t\to\infty}\theta_{\Bar{\kappa}}(t) = - \arctan
        \left(\frac{1}{\kappa \sqrt{\lambda}}\right) + k \pi \qquad k \in \mathbb{Z}.
    \end{align*}
    However, $k=1$, since otherwise $\theta_{\Bar{\kappa}}$ must cross either $0$ or $\pi$, in which case \eqref{eqq::x} would be in Case C.
    Assume there exists a $\kappa < \Bar{\kappa}$ such that $x_\kappa$ is in Case B and recall from (i) that $\theta_\kappa(t) \geq \theta_{\Bar{\kappa}}(t)$ for all $t\geq 0$. This implies that 
    \begin{align*}
        - \arctan \left(\frac{1}{\Bar{\kappa} \sqrt{\lambda}}\right) + \pi 
        =\lim_{t\to \infty} \theta_{\Bar{\kappa}}(t)
        \leq \lim_{t\to \infty} \theta_\kappa (t)
        = - \arctan \left(\frac{1}{\kappa \sqrt{\lambda}}\right) + \pi
    \end{align*}
    contradicting 
    \begin{align*}
        - \arctan \left(\frac{1}{\Bar{\kappa} \sqrt{\lambda}}\right) > - \arctan \left(\frac{1}{\kappa \sqrt{\lambda}}\right)
    \end{align*}
    for $\kappa < \Bar{\kappa}$.
\end{proof}

\begin{remark}
The assumption $\attr^x(0) > 0$ is necessary here, since the initial condition of \eqref{eqq::proof::thetatilde} is otherwise not any longer decreasing in $\kappa$. However, we expect that extensions to $\attr^x(0) \in \R$ are possible, e.g.~by considering the time reversed system (then one gains an additional negative sign in the initial condition) and exploiting that a trajectory can only be hyperbolic for \eqref{eqq::x} if it is hyperbolic for the time reversed system.
\end{remark}

In the following corollary, we combine the results obtained in the previous theorems to obtain our main existence result of a unique critical rate, where \eqref{eqq::x} is in Case B, which separates Cases A and C.

\begin{corollary}\label{cor::transversalTipping}
    Assume that $\alpha\in \SET(\R)$ is such that the attractor $\attr^x$ of \eqref{eqq::x}, guaranteed by \autoref{thm::casesx}, satisfies $\attr^x(0) > 0$. If $g$ is such that there exists a $\kappa_1>0$ such that \eqref{eqq::x}$_{\kappa_1}$ is in Case C, then there exists a unique $\Bar{\kappa} > \kappa_1$ with:
    \begin{itemize}
        \item \eqref{eqq::x}$_{\kappa}$ is in Case A for $\kappa > \Bar{\kappa}$;
        \item \eqref{eqq::x}$_{\kappa}$ is in Case B for $\kappa = \Bar{\kappa}$;
        \item \eqref{eqq::x}$_{\kappa}$ is in Case C for $\kappa < \Bar{\kappa}$.
    \end{itemize}
    i.e.~there exists a unique transversal transition from Case A to C via Case B at rate $\Bar{\kappa}$.
\end{corollary}

\begin{proof}
    By \autoref{thm::largekappa} (i), there exists $\kappa_2$ such that \eqref{eqq::x}$_\kappa$ is in Case A for $\kappa > \kappa_2$. Therefore, by \autoref{thm::casesopen}, there is a transition from Case A to C. By \autoref{thm::transversaltipping} (i), there is only one such transition. By \autoref{thm::casesopen}, the existence of such a transition implies that there exists a $\Bar{\kappa}$ such that \eqref{eqq::x}$_{\Bar{\kappa}}$ is in Case B. By \autoref{thm::transversaltipping} (ii) this $\Bar{\kappa}$ is unique.
\end{proof}

\begin{remark}
By \autoref{pro::xysamecase}, the same holds for \eqref{eqq::y} with $\attr^y(0) > 0 + \int_{-\infty}^0 \alpha(t) \dd t$, however we only need the result for \eqref{eqq::x}. One might wonder how large the group of functions $g$ is where \autoref{cor::transversalTipping} can be applied, i.e.~where there exist $\kappa_1$ such that \eqref{eqq::x}$_{\kappa_1}$ is in Case C. One can prove that this at least includes all $g \in \SET(\R^+)$ with $l\{t: g(t) >\lambda\}>0$ (see Appendix \autoref{pro::smallKappa_underCond}).
\end{remark}

\autoref{cor::transversalTipping} allows us to define a function $\Kappa$ that maps $g\in \SET(\R^+)$ to the corresponding unique critical rate $\Bar{\kappa}$, whenever the \autoref{cor::transversalTipping} can be applied.

\begin{definition}\label{def::Kappa}
Given $\alpha$ such that $\attr^x (0) > 0$, we define the \textit{rate function}
\begin{align*}
    \Kappa: (0,\infty) \times \SET(\R) \times \SET(\R^+) \longrightarrow (0,\infty), \ \Kappa(\lambda, \alpha, g) = \left\{\begin{array}{cc}
       \Bar{\kappa}(\lambda, \alpha,g)  &  \text{\autoref{cor::transversalTipping} applies}\\
       1  &  \text{else,}
    \end{array} \right. 
\end{align*}
where $\Bar{\kappa} = \Bar{\kappa}(\lambda, \alpha,g)$ is the unique transition rate from \autoref{cor::transversalTipping} such that \eqref{eqq::x}$_{\Bar{\kappa}}$ is in Case B.
\end{definition}

Note that \autoref{def::Kappa} ensures that the attractor $\attr^x_{\Kappa(\lambda, \alpha,g)}$ of the x-system \eqref{eqq::x}
\begin{align*}
    \dot{x}(t) = \lambda - h(t; \Kappa(\lambda, \alpha,g)) - x^2(t)
\end{align*}
is 
\begin{itemize}
    \item in Case B, whenever \autoref{cor::transversalTipping} applies,
    \item in Case A or B, whenever \autoref{cor::transversalTipping} does not apply.
\end{itemize}
In particular, $\attr^x_{\Kappa(\lambda, \alpha,g)}$ is always bounded.

\clearpage
\section{Reformulation of the optimal control problem} \label{sec::reformOCproblem}

Let $\lambda >0$ and $\alpha \in \SET(\R)$ such that the solution to the uncontrolled equation
\begin{align}\label{eqq::sec3uncontrolled}
    \dot{x}(t) = \lambda - \alpha(t) - x^2(t) \qquad x(0) = \mu^x
\end{align}
does not exist on the whole of $\R$.
The optimal control problem that we consider is 
\begin{align}\tag{\textbf{OCP}}%\label{prob::OCP}
        \min_{u \in \SET(\R^+)} \int_0^\infty u^2(t) \dd t \qquad s.t.~
        \left\{ 
            \begin{array}{lll}
            &\dot{x}(t) = \lambda - \alpha(t) + u(t) - x^2(t)  \\
            & x(0) = \mu^x >0 \\
            &\exists  \  M \in \N \ s.t.~\ x(t) \geq -M \quad \forall \  t \in \R.
            \end{array}
        \right.
\end{align}
Note that our assumption on $\alpha$ guarantee that a control is necessary in the first place. We further define the \textit{controlled equation}
\begin{align}\label{eqq::sec3controlled}
    \dot{x}(t) = \lambda - \alpha(t) + u(t) - x^2(t).
\end{align}

The nonautonomous theory developed in Section \ref{sec::nonautonomous_fold} concerns itself with the properties of the pullback attractor on the whole real line and therefore, in particular, with the dynamics on $(-\infty,0)$. This stands in contrast to optimal control, where one observes a state $\mu^x$ at $t=0$ and the goal is to control it on $[0,\infty)$. To still be able to apply the theory developed in Section \ref{sec::nonautonomous_fold} one needs to connect both viewpoints. In that regard, note that for any $\mu^x > 0$ the trajectory $x(t,0,\mu^x)$ solving
\begin{align*}
    \dot{x} = \lambda - \alpha(t) + u(t) - x^2 \qquad x(0)=\mu^x
\end{align*}
 can be identified with the pullback attractor of 
 \begin{align*}
    \dot{x} = \lambda - \alpha(t) + u(t) - x^2 \qquad s.t. \qquad \left\{\begin{array}{cc}
       \textit{supp}(u) \subset [0,\infty)  &  \\
        \alpha(t) = \lambda + 1  & t \in [-T,0),
    \end{array}\right. 
\end{align*}
where $T = (\arctan(\sqrt{\lambda}) - \arctan(\mu^x))$.
To see this, first note that if $\alpha$ is supported on $[-T,\infty)$, then $\attr^x$ is the trajectory with $\attr^x(-T) = \sqrt{\lambda}$. This is a consequence of $\attr^x$ being the only solution that goes to $\sqrt{\lambda}$ as $t\to-\infty$ (cf. \autoref{thm::casesx}). Further, the solution of
\begin{align*}
    \dot{z} = -1 -z^2 \qquad z(-T) = \sqrt{\lambda}
\end{align*}
is given by
\begin{align*}
    z(t) = \tan(\arctan(\sqrt{\lambda}) - (T+t)).
\end{align*}
It follows that $z(0) = \mu^x$. Combined with the uniqueness of solutions, this shows the claim.
Keeping this in mind, we can, without loss of generality, use the statements of Section \ref{sec::nonautonomous_fold} for trajectories $x(t)$ with $x(0)=\mu^x >0$. In particular, any statements in Sections \ref{sec::reformOCproblem} and \ref{sec::analyticalsol} are written in terms of trajectories of \eqref{eqq::sec3controlled} with $x(0) = \mu^x$ and not in terms of pullback attractors.
\\
\\
This section is split into two parts: In the first part, Section \ref{sec::reformsubsec}, we reformulate \eqref{prob::OCP} into an unconstrained minimization problem using the rate function $\Kappa$ from \autoref{def::Kappa}. 

Section \ref{sec::rateApproxTheory} is devoted to the study of how the optimal cost obtained changes with approximations of $\Kappa$. We find that it is continuous and, under mild assumptions, even locally Lipschitz in the approximation. 

We start by showing that \eqref{prob::OCP} has a solution.

\begin{proposition}\label{pro::ocp_has_solution}
    The optimal control problem \eqref{prob::OCP} has a solution.
\end{proposition}

\begin{proof}
    We show that \eqref{prob::OCP} is equivalent to the problem 
    \begin{align} \label{prob::bounded}
    \min_{u\in \SET(\R^+)} \int_0^\infty u^2(t) \dd t \qquad s.t.~\qquad \left\{\begin{array}{lll}
    &\dot{x}(t) = \lambda - \alpha(t) + u(t) - x^2(t) \\
    &x(0)= \mu^x > 0 \\
     &x \geq -M
\end{array}\right.
\end{align}
for some $M>0$. The latter has a solution by \cite[Thm.~1]{existenceofsolutions}. \\ 
First, note that $u_{triv} = \alpha$ is an admissible control with cost $||\alpha||_{2;[0,\infty)}^2$. Moreover, we define the auxiliary function 
$f(s) = e^s/ \sqrt{s}$ which has a minimum at $s=1/2$ with minimal value $C>0$. Using this, we define
\begin{align*}
    M \coloneqq \lambda + ||\alpha||_{\infty;[0,\infty)} + \sqrt{\lambda} + 1+ \frac{(||\alpha||_{2;[0,\infty)} + 1)^\frac{1}{2}}{C}.
\end{align*}
Assume $u$ is an admissible control such that there exists $t_0>0$ with $x(t_0) = -M$. Since $u$ is feasible, it follows by \autoref{thm::casesx} that $x\to\pm\sqrt{\lambda}$. In particular, this means there exists $t>t_0$ such that $x(t) = -(\lambda + ||\alpha||_{\infty;[0,\infty)} + \sqrt{\lambda} +1)$. Further, we have \
\begin{align*}
    \dot{x} \leq \lambda + ||\alpha||_{\infty;[0,\infty)} + u + x \qquad \forall \  x\in[-1,-\infty)
\end{align*}
and therefore by Gronwall's inequality
\begin{align*}
    & x(t) \leq -M e^{(t-t_0)} + \int_{t_0}^t e^{(t-t_0)} (\lambda + ||\alpha||_{\infty;[0,\infty)} + u(s)) \dd s \\
    \Longrightarrow \  &- (\lambda + ||\alpha||_{\infty;[0,\infty)} +\sqrt{\lambda} +1) \leq -M e^{(t-t_0)} + \int_{t_0}^t e^{(t-t_0)} (\lambda + ||\alpha||_{\infty;[0,\infty)} + u(s)) \dd s \\
    \Longrightarrow \ &- (\lambda + ||\alpha||_{\infty;[0,\infty)} + \sqrt{\lambda} +1) \leq -M e^{(t-t_0)} + (e^{(t-t_0)}-1) (\lambda + || \alpha||_{\infty;[0,\infty)}) + \int_{t_0}^t u(s) \dd s.
\end{align*}
It follows by Cauchy-Schwarz that
\begin{align*}
    &- (\sqrt{\lambda} +1) + (M-\lambda - ||\alpha||_{\infty;[0,\infty)}) e^{(t-t_0)} \leq \int_{t_0}^t u(s) \dd s \leq ||u||_{1;[0,\infty)]} \leq \sqrt{t - t_0} ||u||_{2;[0,\infty)} \\
    \Longrightarrow \ &(M-\lambda - ||\alpha||_{\infty;[0,\infty)} - (\sqrt{\lambda} +1)e^{-(t-t_0)}) e^{(t-t_0)} \leq \sqrt{t - t_0} ||u||_{2;[0,\infty)}.
\end{align*}
Therefore
\begin{alignat*}{2}
    & \mathrlap{(M-\lambda - ||\alpha||_{\infty;[0,\infty)} - \sqrt{\lambda} -1) e^{(t-t_0)} \leq \sqrt{t - t_0} ||u||_{2;[0,\infty)}}  \\
    \Longrightarrow \ &||u||_{2;[0,\infty)}
    && \geq (M-\lambda - ||\alpha||_{\infty;[0,\infty)}  - \sqrt{\lambda} -1) \frac{e^{(t-t_0)}}{\sqrt{t - t_0}} 
    =  \frac{(||\alpha||_{2;[0,\infty)} + 1)^\frac{1}{2}}{C} f(t-t_0) \\
    & &&{} \geq (||\alpha||_{2;[0,\infty)} + 1)^\frac{1}{2}.
\end{alignat*}
Hence, 
\begin{align*}
    \int_0^\infty u^2(s) \dd s \geq \int_0^\infty \alpha^2(s) \dd s + 1 > \int_0^\infty u_{triv}^2(s) \dd s.
\end{align*}
It follows that any control $u$ that produces $x$-values smaller than $-M$ has a higher cost than the trivial control $u_{triv}$. In particular, the Optimal Control problem \eqref{prob::OCP} is equivalent to \eqref{prob::bounded}, since $x\geq -M$ necessarily needs to be fulfilled for any optimal control. As \eqref{prob::bounded} has a solution by \cite{existenceofsolutions}, it follows that \eqref{prob::OCP} also has a solution. 
\end{proof}

\subsection{Reformulation}\label{sec::reformsubsec}

In this section, we demonstrate how to reformulate \eqref{prob::OCP} to an unconstrained optimization problem. Since \eqref{prob::OCP} has a solution, so does the new unconstrained problem. \\
We start by showing some dynamical properties that an optimal control of \eqref{prob::OCP} must satisfy.

\begin{proposition}\label{pro::caseAnotOptimal}
    Let $x(0) = \mu^x >0$ and $\alpha \in \SET(\R)$. Furthermore, assume that the uncontrolled system \eqref{eqq::sec3uncontrolled} is in Case C. 

    \begin{itemize}
        \item If $\Tilde{u}$ is such that $x_{\Tilde{u}}$ is in Case A, then $\Tilde{u}$ is not optimal for \eqref{prob::OCP}.
    \end{itemize}

\end{proposition}
\begin{proof}
    Since the uncontrolled system is in Case C, this means that there exists a set of positive Lebesgue measure $\mathcal{A}$ such that $\Tilde{u}>0$ on $\mathcal{A}$. Therefore, there exists $\delta>0$ and a set of positive Lebesgue measure $\mathcal{B}$ such that $\Tilde{u}>\delta$ on $\mathcal{B}$, since otherwise
    \begin{align*}
        l(\mathcal{A}) \leq l(\{\Tilde{u} > 0\}) = l\left(\bigcup_{n> 0} \left\{ \Tilde{u} > \frac{1}{n}\right\}  \right)  \leq \sum_{n>0}  l\left( \left\{\Tilde{u} > \frac{1}{n}\right\}\right) = 0.
    \end{align*}
    By \autoref{thm::casesx}, Case A implies that $x_{\Tilde{u}}$ is hyperbolic. By the robustness of hyperbolic solutions under small perturbations (\autoref{pro::robustnesshyperbolic}) there exists $0<\varepsilon<\delta$ such that for 
    \begin{align*}
        \Tilde{u}_\varepsilon = \Tilde{u} - \varepsilon \indicator_{\mathcal{B}}
    \end{align*}
    $x_{\Tilde{u}_\varepsilon}$ is hyperbolic and therefore in Case A. Moreover, $\Tilde{u}_\varepsilon^2 < \Tilde{u}^2$ on $\mathcal{B}$. It follows that
    \begin{align*}
        \int_0^\infty \Tilde{u}_\varepsilon^2(t) \dd t = \int_{\R^+\backslash\mathcal{B}}  \Tilde{u}_\varepsilon^2(t) \dd t + \int_{\mathcal{B}}  \Tilde{u}_\varepsilon^2(t) \dd t < \int_{\R^+\backslash\mathcal{B}}  \Tilde{u}^2(t) \dd t + \int_{\mathcal{B}}  \Tilde{u}^2(t) \dd t = \int_0^\infty \Tilde{u}^2(t) \dd t.
    \end{align*}
    Hence, $\Tilde{u}_\varepsilon$ is a feasible control with a strictly lower cost than $\Tilde{u}$. Therefore, $\Tilde{u}$ cannot be optimal. Note, if necessary, for example, if $\Tilde{u}_\varepsilon$ is not $L^1$, one can restrict to a subset of $\mathcal{B}$ with finite Lebesgue measure.
\end{proof}

\autoref{pro::caseAnotOptimal} has the following consequences: First, an optimally controlled trajectory cannot be in Case C, since it would not be bounded (cf.~\autoref{thm::casesx}). Furthermore, an optimally controlled trajectory cannot be in Case A, since we can otherwise construct an admissible control with lower cost. In particular, any optimally controlled trajectory must be in Case B. As a consequence, we do not change the optimizer by using the function $\Kappa$ to time-rescale every effective forcing $g= \alpha - u$ such that the trajectory is in Case B. However, by the definition of $\Kappa$ (cf.~\autoref{def::Kappa}) any control, $u(t) = \alpha(t) -g(t\Kappa(g))$ obtained by time-rescaling the effective forcing, is admissible. Hence, by the time-rescaling, we restrict to a set that:
\begin{enumerate}[label=(\roman*)]
    \item Is admissible;
    \item Contains the optimizer.
\end{enumerate}
This intuition is made rigorous below.

\begin{problem}[Reformulated Optimal Control Problem]
    Let $x(0) = \mu^x > 0$ and $\alpha\in\SET(\R^+)$ such that the uncontrolled equations \eqref{prob::sec1uncontrolled} satisfies $x \to -\infty$ on $\R^+$. Further, let $\Kappa$ be the rate function from \autoref{def::Kappa}. We consider the unconstrained minimization problem
    \begin{align}\tag{\textbf{RCP}}\label{prob::RCP}
        \min_{g\in\SET(\R^+)} \int_0^\infty [\alpha(t) - g(t \Kappa(\lambda, \alpha, g))]^2 \dd t.
    \end{align}
\end{problem}

\begin{theorem}[Equivalence of Problems]\label{thm::reformulation}
Consider the problems \eqref{prob::OCP} and \eqref{prob::RCP},
then $g^*\in\SET(\R^+)$ is an optimal solution of \eqref{prob::RCP} if and only if
    $u^*(t) = \alpha(t) - g^*(t \Kappa(\lambda,\alpha,g^*))$ 
is an optimal solution of \eqref{prob::OCP}.
\end{theorem}

\begin{proof}
    ``$\Longrightarrow$": Assume that $g^*$ is an optimal solution of \eqref{prob::RCP} and that $u^*(t) = \alpha(t) - g^*(t\Kappa(\lambda,\alpha,g^*))$ is not optimal for \eqref{prob::OCP}. Then, according to \autoref{pro::ocp_has_solution}, there exists $\Tilde{u}\in\SET(\R^+)$ that is optimal for \eqref{prob::OCP} with
    \begin{align*}
        \int_0^\infty \Tilde{u}^2(t)  \dd t < \int_0^\infty (u^*(t))^2  \dd t = \int_0^\infty [\alpha(t) - g^*(t \Kappa(\lambda, \alpha, g^*)]^2 \dd t.
    \end{align*}
    Define 
    \begin{align*}
        \Tilde{g}(t) \coloneqq \alpha(t) - \Tilde{u}(t) \in \SET(\R^+).
    \end{align*}
    By \autoref{pro::caseAnotOptimal} the controlled equation \eqref{eqq::sec3controlled}$_{\Tilde{u}}$ is in Case B and therefore $\Kappa(\lambda, \alpha, \Tilde{g}) = 1$. It follows that 
    \begin{align*}
        &\int_0^\infty [\alpha(t) - \Tilde{g}(t \Kappa(\lambda, \alpha, \Tilde{g}))]^2  \dd t
        = \int_0^\infty [\alpha(t) - \Tilde{g}(t)]^2  \dd t 
        = \int_0^\infty \Tilde{u}^2(t)  \dd t \\
        &< \int_0^\infty (u^*(t))^2  \dd t
        = \int_0^\infty [\alpha(t) - g^*(t \Kappa(\lambda, \alpha, g^*))]^2  \dd t.
    \end{align*}
    This contradicts the optimality of $g^*$. It remains to show that $u^*$ is admissible for \eqref{prob::OCP}. However, by the definition of $\Kappa$, the equation
    \begin{align*}
        \dot{x}(t) = \lambda - g^*(t\Kappa(\lambda,\alpha,g^*)) - x^2(t) = \lambda - \alpha(t) + u^*(t) - x^2(t)
    \end{align*}
    is in Case A or B, and hence $x_{u^*}$ is bounded by \autoref{thm::casesx}. Further, it is easy to see that $u^* \in \SET(\R^+)$.
    \\
    ``$\Longleftarrow$": Assume that $u^*$ is an optimal solution of \eqref{prob::OCP} and define $g^* \coloneqq \alpha - u^*$. Naturally, $g^*\in \SET(\R^+)$. By \autoref{pro::caseAnotOptimal}, \eqref{eqq::sec3controlled}$_{u^*}$ is in Case B and therefore $\Kappa(\lambda, \alpha, g^*) = 1$. Assume there exists $\Tilde{g} \in \SET(\R^+)$ such that
    \begin{align*}
        \int_0^\infty [\alpha(t) - \Tilde{g}(t \Kappa(\lambda, \alpha, \Tilde{g}))]^2 \dd t < \int_0^\infty [\alpha(t) - g^*(t \Kappa(\lambda, \alpha, g^*))]^2 \dd t 
    \end{align*}
    Let $\Tilde{u}(t) \coloneqq \alpha(t) - \Tilde{g}(t\Kappa(\lambda, \alpha, \Tilde{g}))$. By the definition of $\Kappa$, $\Tilde{u}$ is admissible for \eqref{prob::OCP} and further
    \begin{align*}
        \int_0^\infty \Tilde{u}^2(t) \dd t &= \int_0^\infty [\alpha(t) - \Tilde{g}(t \Kappa(\lambda, \alpha, \Tilde{g}))]^2 \dd t < \int_0^\infty [\alpha(t) - g^*(t \Kappa(\lambda, \alpha, g^*))]^2 \dd t \\
        &= \int_0^\infty [\alpha(t) - g^*(t)]^2 \dd t  = \int_0^\infty (u^*(t))^2 \dd t,
    \end{align*}
    contradicting the optimality of $u^*$.
\end{proof}

\begin{remark}\label{rem::projection}
Note that the choice of the cost function is not important for the reformulation in \autoref{thm::reformulation}. In fact, $\int u^2$ can be replaced by any cost function $\mathcal{L}(u)$ that depends solely on $u$ and for which it can be shown that the equivalent of \eqref{prob::OCP} has solutions (e.g.~by checking the assumptions of \cite[Thm.~1]{existenceofsolutions}). However, naturally the choice of the cost function influences whether the solutions are unique. Further, note that while \eqref{prob::OCP} might have a unique solution, \eqref{prob::RCP} does not. This can be easily seen since if $g(t)$ is an optimal solution, then $g_c(t) \coloneqq g(ct)$ for any $c>0$ satisfies
\begin{align*}
    g(t \Kappa(\lambda, \alpha, g)) = g_c(t \Kappa(\lambda,\alpha,g_c))
\end{align*}
and therefore $g_c$ is also an optimal solution to \eqref{prob::RCP}. 
\end{remark}

\subsection{Rate approximation}\label{sec::rateApproxTheory}

\autoref{thm::reformulation} shows that the Optimal Control problem \eqref{prob::OCP} can equivalently be formulated as a nonlinear unconstrained optimization problem \eqref{prob::RCP}. However, this comes at the cost of determining the rate function $\Kappa$. There are multiple promising ways to do this numerically, among them the power series expansion in \cite{Kuehn2022}, the time-compactification arguments in \cite{Tipping_thresholdsWieczorek_2023}, or even a Taylor expansion. Furthermore, one might also be interested in implementing a cheaper but less precise expression for $\Kappa$ to balance precision and computation efficiency. In any case, it is important to know whether the cost one obtains for any such approximation behaves well in the ``goodness" of the approximation. In the following, we find that the optimal cost function value is continuous in the approximation and in some cases even locally Lipschitz.

\begin{definition}\label{def::feasibleRateFunction}
Similarly to \autoref{def::Kappa}, a function 
$\Tilde{\Kappa} : (0,\infty) \times \SET(\R) \times \SET(\R^+) \to (0,\infty)$ is a \textit{feasible rate function} if, for any set of parameters $\{\lambda, \alpha,g\}$ and $h$ as defined in \eqref{eqq::defh}, the equation
\begin{align*}
    \dot{x}(t) = \lambda - h(t;\Tilde{\Kappa}(\lambda,\alpha,g)) - x^2(t)
\end{align*}
is in Case A or B.

Given a feasible rate function $\Tilde{\Kappa}$, e.g.~as an approximation of $\Kappa$, it is natural to consider the minimization problem
\begin{align}\label{prob::sec3approximatedProblem}
    \min_{g\in\SET(\R^+)} \int_0^\infty [\alpha(t) - g(t\Tilde{\Kappa}(\lambda, \alpha, g))]^2 \dd t
\end{align}
we call \eqref{prob::sec3approximatedProblem} the \textit{approximated OCP} and a solution to \eqref{prob::sec3approximatedProblem} (if existent) an \textit{approximated solution}.
\end{definition}

 Note that $\Kappa$ from \autoref{def::Kappa} is clearly a feasible rate function. 

\begin{remark}
The definition of \eqref{prob::sec3approximatedProblem} opens the question, under what conditions on $\Tilde{\Kappa}$ an approximated OCP has a solution. Research in this regard is postponed to future papers, and for now we simply assume that any feasible rate approximation $\Tilde{\Kappa}$, we consider, is such that the approximated OCP has a solution.    
\end{remark}

\begin{remark}\label{rem::apprx_g_admissable}
Note that by the same steps as in \autoref{thm::reformulation}, one can show that if $\Tilde{g}$ is an optimal solution to \eqref{prob::sec3approximatedProblem} then $\Tilde{u} = \alpha - \Tilde{g}(\cdot \Tilde{\Kappa}(\lambda, \alpha, \Tilde{g})) $ is an admissible solution to \eqref{prob::OCP}, with cost $\int_0^\infty \Tilde{u}(t) \dd t = \int_0^\infty [\alpha(t) - g(t \Tilde{\Kappa}(\lambda, \alpha, g))]^2 \dd t$.
\end{remark}

The following theorem shows that the optimal value is continuous, and under mild conditions even Lipschitz in the rate function $\Kappa$.

\begin{theorem}\label{thm::continuity}
    Let $\alpha\in\SET(\R^+)$ such that the solution to the uncontrolled problem \eqref{eqq::sec3uncontrolled} diverges in finite time and let $g^*$ be an optimal solution to \eqref{prob::RCP}.
    \begin{enumerate}[label=(\roman*)]
        \item Let $(\Kappa_n)_{n\in\N}$ be a sequence of feasible rate functions such that $|\Kappa_n(\lambda, \alpha,g^*) - \Kappa(\lambda, \alpha,g^*)| \to 0$ as $n\to\infty$. The optimal value is \textit{continuous} in the rate function, i.e.~
        \begin{align*}
        \min_{g_n\in\SET(\R^+)} ||\alpha - g_n(\cdot \Kappa_n(\lambda, \alpha, g_n))||_{2;[0,\infty)}^2 \longrightarrow ||\alpha - g^*(\cdot \Kappa(\lambda, \alpha, g^*))||_{2;[0,\infty)}^2.
        \end{align*}

        \item Assume additionally that $g^*\in BV(\R^+)$ with $t(g^*)'\in L^2(\R^+)$. Further, let $0< C \coloneqq  \Kappa(\lambda, \alpha, g^*)$,
        then there exists $L=L(\alpha, g^*, C)>0$ such that
        \begin{align*}
            \bigg|\min_{g\in\SET(\R^+)} ||\alpha - g(\cdot \Tilde{\Kappa}(\lambda, \alpha, g))||_{2;[0,\infty)}^2 -  ||\alpha - g^*(\cdot \Kappa(\lambda, \alpha, g^*))||_{2;[0,\infty)}^2 \bigg| \leq L |\Tilde{\Kappa}(\lambda, \alpha,g^*) - \Kappa(\lambda, \alpha,g^*)|
        \end{align*}
        i.e.~the optimal value obeys a type of \textit{global Lipschitz bound} in the rate function.
    \end{enumerate}
\end{theorem}

For the proof of \autoref{thm::continuity}, we need the following Lemma, the proof of which has been moved to the appendix.

\begin{lemma}\label{lem::propertieslossfunc}
     The following holds:
    Fix $\alpha \in \SET(\R), \ g \in \SET(\R^+)$, then
    \begin{enumerate}[label=(\roman*)]
        \item The map $\R^+ \ni \kappa \mapsto ||\alpha - g(\cdot \kappa)||_{2;[0,\infty)}^2$ is continuous;
        \item Let $I \subset [\kappa_-,\kappa_+)$ for $0<\kappa_-<\kappa_+\leq \infty$. If additionally $g \in BV(\R^+)$ with $tg'\in L^2(\R^+)$, then there exists $L=L(\alpha,g,I)$ such that
        \begin{align*}
            \bigg| ||\alpha - g(\cdot \kappa_1)||_{2;[0,\infty)}^2 - ||\alpha - g(\cdot \kappa_2)||_{2;[0,\infty)}^2 \bigg| \leq L |\kappa_1 - \kappa_2| \qquad \forall \  \kappa_1,\kappa_2 \in I.
        \end{align*}
    \end{enumerate} 
\end{lemma}

\begin{proof}[Proof of \autoref{thm::continuity}]
    (i) Let $(g_n^*)_{n\in\N}$ be solutions to 
    \begin{align}\label{eqq::proof::gn}    
    \min_{g_n\in\SET(\R^+)} ||\alpha - g_n(\cdot \Kappa_n(\lambda, \alpha, g_n))||_{2;[0,\infty)}^2.
    \end{align}
    Note that 
    \begin{align}\label{eqq::proof::origKappaBetter}
        ||\alpha - g_n^*(\cdot \Kappa_n(\lambda,\alpha,g_n^*))||_{2;[0,\infty)}^2 \geq ||\alpha - g^*(\cdot\Kappa(\lambda, \alpha, g^*))||_{2;[0,\infty)}^2
    \end{align}
    for all $n\in\N$, since otherwise, by the same steps as in the proof of \autoref{thm::reformulation}, $u^*_n(t) \coloneqq \alpha(t) - g_n^*(t \Kappa_n(\lambda, \alpha,g_n^*))$ would be an admissible control for \eqref{prob::OCP} with strictly lower cost than $u^*(t) \coloneqq \alpha(t) - g^*(t \Kappa(\lambda, \alpha,g^*))$. This contradicts the statement of \autoref{thm::reformulation}.
    \\\\
    Fix $\varepsilon>0$. By \autoref{lem::propertieslossfunc} (i), there exists $\delta>0$ such that for all $\kappa,\kappa_n>0$ with $|\kappa - \kappa_n| < \delta$
    \begin{align}\label{eqq::proof::continuityInKappa}
        \bigg|||\alpha - g^*(\cdot\kappa)||_{2;[0,\infty)}^2 - ||\alpha -g^*(\cdot\kappa_n)||_{2;[0,\infty)}^2\bigg| < \varepsilon.
    \end{align}
    Further, by the assumptions on $(\Kappa_n)_{n\in\N}$ there exists $n_0\in\N$ such that for all $n\geq n_0$, \\
    $|\Kappa(\lambda, \alpha,g^*) - \Kappa_n(\lambda, \alpha,g^*)| < \delta$. Hence, for all $n\geq n_0$: 
    \begin{align*}
        ||\alpha - g_n^*(\cdot\Kappa_n(\lambda,\alpha,g_n^*))||_{2;[0,\infty)}^2& \geq ||\alpha - g^*(\cdot\Kappa(\lambda,\alpha,g^*))||_{2;[0,\infty)}^2 \geq ||\alpha - g^*(\cdot\Kappa_n(\lambda,\alpha,g^*))||_{2;[0,\infty)}^2 - \varepsilon \\
        &\geq ||\alpha - g_n^*(\cdot\Kappa_n(\lambda,\alpha,g_n^*))||_{2;[0,\infty)}^2 - \varepsilon, 
    \end{align*}
    where the first inequality comes from \eqref{eqq::proof::origKappaBetter}, the second from the continuity \eqref{eqq::proof::continuityInKappa} and the third inequality follows from the optimality of $g_n^*$ for the problem \eqref{eqq::proof::gn}. Therefore,
    \begin{align*}
        \varepsilon & \geq  ||\alpha - g_n^*(\cdot\Kappa_n(\lambda,\alpha,g_n^*))||_{2;[0,\infty)}^2 - ||\alpha - g^*(\cdot\Kappa(\lambda,\alpha,g^*))||_{2;[0,\infty)}^2 \\
        &= \bigg| ||\alpha - g_n^*(\cdot\Kappa_n(\lambda,\alpha,g_n^*))||_{2;[0,\infty)}^2 - ||\alpha - g^*(\cdot\Kappa(\lambda,\alpha,g^*))||_{2;[0,\infty)}^2 \bigg|
    \end{align*}
    for all $n\geq n_0$, which proofs the claim.
    \\\\
    (ii) Let additionally $g^*\in BV$, $t(g^*)' \in L^2$ and $C\coloneqq \Kappa(\lambda,\alpha,g^*)$. Note that $\Kappa(\lambda,\alpha,g^*) \leq \Tilde{\Kappa}(\lambda,\alpha,g^*)$ and therefore 
    \autoref{lem::propertieslossfunc} (ii) can be applied with $I = [C,\infty)$. Hence, there exists $L=L(\alpha,g^*,C)$ such that 
    \begin{align*}
        \bigg| ||\alpha - g^*(\cdot \Tilde{\Kappa}(\lambda, \alpha, g^*)) ||_{2;[0,\infty)}^2 - ||\alpha - g^*(\cdot \Kappa(\lambda, \alpha, g^*)) ||_{2;[0,\infty)}^2 \bigg| 
        &\leq L \big| \Tilde{\Kappa}(\lambda, \alpha, g^*) - \Kappa(\lambda, \alpha, g^*) \big|.
    \end{align*}
    Let $\Tilde{g}^*$ be a solution to $ 
    \min_{g\in\SET(\R^+)} ||\alpha - g(\cdot \Tilde{\Kappa}(\lambda, \alpha, g))||_{2;[0,\infty)}^2
    $.
    By the same steps as in the proof of (i), we get
    \begin{align*}
        ||\alpha - \Tilde{g}^*(\cdot\Tilde{\Kappa}(\lambda,\alpha,\Tilde{g}^*))||_{2;[0,\infty)}^2& \geq ||\alpha - g^*(\cdot\Kappa(\lambda,\alpha,g^*))||_{2;[0,\infty)}^2 \\
        &\geq ||\alpha - g^*(\cdot\Tilde{\Kappa}(\lambda,\alpha,g^*))||_{2;[0,\infty)}^2 - L |\Tilde{\Kappa}(\lambda, \alpha,g^*) - \Kappa(\lambda, \alpha,g^*)| \\
        &\geq ||\alpha - \Tilde{g}^*(\cdot\Tilde{\Kappa}(\lambda,\alpha,\Tilde{g}^*))||_{2;[0,\infty)}^2 - L |\Tilde{\Kappa}(\lambda, \alpha,g^*) - \Kappa(\lambda, \alpha,g^*)|,
    \end{align*}
    and therefore 
    \begin{align*}
        L |\Tilde{\Kappa}(\lambda, \alpha,g^*) - \Kappa(\lambda, \alpha,g^*)|
        &\geq ||\alpha - \Tilde{g}^*(\cdot\Tilde{\Kappa}(\lambda,\alpha,\Tilde{g}^*))||_{2;[0,\infty)}^2 - ||\alpha - g^*(\cdot\Kappa(\lambda,\alpha,g^*))||_{2;[0,\infty)}^2 \\
        &= \bigg| ||\alpha - \Tilde{g}^*(\cdot\Tilde{\Kappa}(\lambda,\alpha,\Tilde{g}^*))||_{2;[0,\infty)}^2 - ||\alpha - g^*(\cdot\Kappa(\lambda,\alpha,g^*))||_{2;[0,\infty)}^2\bigg|.
    \end{align*}
\end{proof}

\begin{remark}    
Note, we do not claim $g_n^* \to g^*$ in any norm. Further, \autoref{thm::continuity} shows that the rate approximation $\Tilde{K}$ only needs to be good at the optimal solution $g^*$. This can be used in numerical implementations. In particular, if one has prior information about $g^*$ one can use this to refine a rate approximation around this prior.
\end{remark}

\clearpage

\section{Closed form solution through approximated rate function}\label{sec::analyticalsol}

In the following, we present how a feasible rate function, i.e.~an approximation of $\Kappa$ can lead to a closed form expression of the solution to the approximated OCP \eqref{prob::sec3approximatedProblem}. It is important to note that by the definition of a feasible rate function \autoref{def::feasibleRateFunction}, the closed form expression of the approximated control avoids tipping.
\\
This section is split into two parts. In the first, we present one such approximation, which then leads to a closed form expression of the solution to the approximated OCP \eqref{prob::sec3approximatedProblem}. Following the discussion in Section \ref{sec::rateApproxTheory} this leads to a closed form expression of an admissible control for \eqref{prob::OCP}. Naturally, this approximated control has a higher cost than the optimal control. Note, one can show that the control one obtains in this way is guaranteed to perform strictly better than the trivial control $u_{triv} = \alpha$. 

A closed form expression for an admissible control is of major importance in situations where controls have to be calculated frequently (e.g.~for uncertainty quantification) or in scenarios when computation time is of utmost importance.
In the second part of this section, we compare the obtained approximated control $\Tilde{u}(t) = \alpha(t) - \Tilde{g}(t\Tilde{K}(\Tilde{g}))$ for \eqref{prob::OCP} with the numerical solution obtained by passing the problem as a constrained nonlinear optimization problem to IPOPT (\cite{ipopt}). Comparisons are made in execution time, objective function values, and shape of the controls and the resulting trajectories.

\subsection{Derivation of closed form control for the approximated rate}

Let $\lambda>0$, $\alpha \in \SET(\R)$ such that the solution of 
\begin{align*}
    \dot{x}(t) = \lambda - \alpha(t) - x(t)^2 \qquad x(0) = \mu^x \in (-\sqrt{\lambda}, \sqrt{\lambda})
\end{align*}
diverges to negative infinity on $(0,\infty)$.

Note, since this section does not use the uniqueness of the critical rate (\autoref{thm::transversaltipping}) one can loosen the assumption $x(0) > 0$ of the foregoing sections. However, we require two additional assumptions in this section, namely that $\alpha,g \geq 0$. The former means that the given external forcing is solely destabilizing the system and never stabilizing, while the latter means that the control is not allowed to stabilize the system beyond undoing the external forcing completely (recall $g(t\kappa) = \alpha(t) - u(t)$). For simplicity, we additionally assume $x(0) \in (-\sqrt{\lambda}, \sqrt{\lambda})$ to have a more concise formulation of the feasible rate function.
\\\\
By \autoref{thm::largekappa} (ii) the function 
\begin{align*}
\Tilde{\Kappa} : (0,\infty) \times \SET(\R) \times \SET(\R^+), \ \Tilde{\Kappa} (\lambda, \alpha, g) \coloneqq \frac{||g||_{1;[0,\infty)}}{\sqrt{\lambda} + x(0)} = \frac{||g||_{1;[0,\infty)}}{m_{\lambda,\alpha}},
\end{align*}
where $m_{\lambda,\alpha} \coloneqq \lambda + x(0)$, is a feasible rate approximation. In the following, we use this rate and consider the approximated control problem 
\begin{align}\label{prob::sec4approx}
    \min_{g\in\SET(\R^+), g\geq 0} \int_0^\infty \left[\alpha(t) - g\left(t \frac{||g||_{1;[0,\infty)}}{m_{\lambda,\alpha}}\right)\right]^2 \dd t .
\end{align}

\begin{theorem}\label{thm::approximatedcontrol}
    Let $(f)^+ $ denote the positive part of $f$ and choose $\Tilde{\nu} \geq 0$ such that \\
    $\int_0^\infty  \big(\alpha(sm_{\lambda, \alpha}) - \frac{\Tilde{\nu}}{2}\big)^+ \dd s = 1$. Then:
    \begin{enumerate}[label=(\roman*)]
        \item The function $\Tilde{g}(s) = \big(\alpha(sm_{\lambda, \alpha}) - \frac{\Tilde{\nu}}{2}\big)^+$
            is a solution to \eqref{prob::sec4approx};
        \item The control
                $\Tilde{u}(s) = \alpha(s) - \big(\alpha(s) - \frac{\Tilde{\nu}}{2} \big)^+ $
            is an admissible control for the original problem \eqref{prob::OCP}, corresponding to $\Tilde{g}$ through $\Tilde{u}(s) = \alpha(s) - \Tilde{g}(s\Tilde{\Kappa}(\lambda, \alpha, \Tilde{g}))$.
    \end{enumerate}

\end{theorem}

\vspace{1cm}

\begin{proof}[Proof of \autoref{thm::approximatedcontrol}]
    Consider the problem \eqref{prob::sec4approx}. Upon the time-rescaling $s = t/ m_{\lambda,\alpha}$, one obtains
    \begin{align*}
        \min_{g\in\SET, g\geq 0}m_{\lambda,\alpha} \int_0^\infty \left[\alpha(sm_{\lambda,\alpha}) - g\left(s ||g||_{1;[0,\infty)}\right)\right]^2 \dd s.
    \end{align*}
    Naturally, the prefactor $m_{\lambda,\alpha}$ does not influence the optimizer. Furthermore, note that $g\geq0$ implies $||g(\cdot ||g||_{1;[0,\infty)})||_{1;[0,\infty)} =1$. Therefore, \eqref{prob::sec4approx} can be equivalently written as 
    \begin{align}\label{prob::sec4approx_proj}
        \min_{g\in\SET(\R^+), g\geq 0, ||g||_{1,[0,\infty)} = 1} \int_0^\infty \left[\alpha(s m_{\lambda,\alpha}) - g(s) \right]^2 \dd t .
    \end{align}

    To find the minimizer of \eqref{prob::sec4approx_proj} we use infinite dimensional Lagrange multipliers and in particular the theory from \cite{infiniteDimensionalLagrangeMultiplier}.
    For that, we consider $\SET(\R^+)$ as a convex subspace of $L^2(\R^+)$, the convex cone $C \coloneqq \{g\in L^2(\R^+) : g\geq 0\} \subset L^2(\R^+)$ and its dual
    \begin{align*}
        C^* \coloneqq \{p \in L^2 : \langle p, g \rangle \geq 0 \ \forall \ g\in C \}    
    \end{align*}
    Note that $C^* = C$, and define
    \begin{align*}
    & f: \SET(\R^+) \to \R, \ f(g) \coloneqq \int_0^\infty [\alpha(sm_{\lambda,\alpha}) - g( s)]^2 \dd s && \text{is convex}, \\
    & f_1: \SET(\R^+) \to L^2, \ f_1(g) \coloneqq -g && \text{is convex, and} \\
    & f_2: \SET(\R^+) \to \R, \ f_2(g) \coloneqq \int_0^\infty g(s) \dd s - 1&& \text{affine linear}.
\end{align*}
Thus \eqref{prob::sec4approx_proj} becomes
\begin{align*}
    \min_g f(g) \qquad s.t.~\left\{\begin{array}{ccc}
        g \in \SET(\R^+) &  \\
        f_1(g) \in -C & \\
        f_2(g) = 0.
    \end{array}\right.
\end{align*}
We define the Lagrangian
\begin{align*}
    \mathcal{L}: \SET(\R^+) \times C \times \R \to \R, \ \mathcal{L}(g,\mu,\nu) = f(g) + \langle \mu, f_1(g)\rangle_{L^2} + \langle\nu, f_2(g)\rangle_\R.
\end{align*}
Let $\mu\in C,\  \nu \in \R$ and $g,g_0\in \SET(\R^+)$, and note
\begin{align*}
    & f'(g_0) (g-g_0) = \int_0^\infty 2(g_0(s) - \alpha(sm_{\lambda,\alpha})) (g(s) - g_0(s)) \dd s \\
    &\langle \mu, f_1'(g_0) (g-g_0) \rangle_{L^2}= \int_0^\infty -\mu(s) (g(s)-g_0(s)) \dd s \\
    & \langle \nu, f_2'(g_0) (g-g_0)\rangle_\R = \int_0^\infty \nu (g(s) - g_0(s)) \dd s,
\end{align*}
where the derivatives are directional derivatives. Hence,
\begin{align*}
    \mathcal{L}'(g_0,\mu,\nu) (g-g_0) 
    &= \big(f'(g_0) + \langle \mu, f_1'(g_0)\rangle_{L^2} + \langle \nu, f_2'(g_0) \rangle_\R\big) ( g- g_0) \\
    &= \int_0^\infty \big[2(g_0(s) - \alpha(s m_{\lambda, \alpha})) - \mu(s) + \nu\big] (g(s)-g_0(s))  \dd s.
\end{align*}
The necessary optimality condition (\cite{infiniteDimensionalLagrangeMultiplier}) is $\mathcal{L}'(g_0,\mu,\nu)(g-g_0) \geq 0$ for all $g\in\SET(\R^+)$. With that in mind, let
\begin{align*}
    &\Tilde{g} \coloneqq \big(\alpha(sm_{\lambda, \alpha}) - \frac{\Tilde{\nu}}{2}\big)^+ 
    & \Tilde{\mu} \coloneqq 2\big( \frac{\Tilde{\nu}}{2} - \alpha(s m_{\lambda, \alpha})\big)^+
    && \Tilde{\nu} \ \text{such that} \ \int_0^\infty \Tilde{g}(s) \dd s = 1.
\end{align*}
Note, such a $\Tilde{\nu}$ exists by Appendix \autoref{lem::integral1nu}. It follows that $\Tilde{g}\in\SET(\R^+)$, $g\geq 0$, $||\Tilde{g}||_1 =1$, $\Tilde{\mu}\in C$ and $\langle \Tilde{\mu}, \Tilde{g} \rangle_{L^2} = 0$. Moreover, 
\begin{align*}
    &\mathcal{L}'(\Tilde{g},\Tilde{\mu},\Tilde{\nu}) (g-\Tilde{g})  = \\
    &= \int_0^\infty\left(2\left(\big(\alpha(sm_{\lambda, \alpha}) - \frac{\Tilde{\nu}}{2}\big)^+ - \alpha(s m_{\lambda, \alpha})\right) - 2\left( \frac{\Tilde{\nu}}{2} - \alpha(s m_{\lambda, \alpha})\right)^+  + \Tilde{\nu} \right) (g(s)-\Tilde{g}(s)) \dd s \\
    &= 0
\end{align*}
for any $g\in \SET(\R^+)$. By \cite[Thm.~3]{infiniteDimensionalLagrangeMultiplier}, it follows that $\Tilde{g}$ is an optimal solution to \eqref{prob::sec4approx_proj} and thus for \eqref{prob::sec4approx}. If $\Tilde{g}$ is an optimal solution to \eqref{prob::sec4approx}, i.e.~the approximated OCP, then it follows by \autoref{rem::apprx_g_admissable} that $\Tilde{u}(t) =\alpha(t) - \Tilde{g}(t \Tilde{K}(\lambda, \alpha, \Tilde{g}))$ is an admissible control to the original problem \eqref{prob::OCP}. Moreover, $||\Tilde{g}||_{1;[0,\infty)} = 1$ and therefore
\begin{align*}
    \Tilde{u}(t) &= \alpha(t) - \Tilde{g}(t \Tilde{K}(\lambda, \alpha, \Tilde{g})) 
    = \alpha(t) - \Tilde{g}\left(t\frac{||\Tilde{g}||_{1;[0,\infty)}}{m_{\lambda, \alpha}}\right)
    = \alpha(t) - \left( \alpha\left(t \frac{m_{\lambda, \alpha}}{m_{\lambda, \alpha}} \right) - \frac{\Tilde{\nu}}{2}  \right)^+ \\
    &= \alpha(t) - \left(\alpha(t) - \frac{\Tilde{\nu}}{2} \right)^+.
\end{align*}
\end{proof}

\subsection{Numerical validation of optimal control approximation}\label{sec::numerics}

In this section, we analyse how the approximated solution of \eqref{prob::OCP} derived in \autoref{thm::approximatedcontrol} performs, compared to a numerically computed solution. Comparisons are made in execution time, cost function value, shape of the controls and shape of the resulting trajectories.
For the numerical solution, we first discretize the dynamics and then write the optimal control problem as a constrained nonlinear optimization problem using CasADi (\cite{Casadi}). We use a multiple shooting method for the dynamics and the rockit library (\cite{rockit}) for easier implementation in CasADi. The  constrained nonlinear optimization problem is solved using IPOPT (\cite{ipopt}).
Note that since the loss landscape of the discretized dynamics might become arbitrarily complicated, the numerical solution might only be a local minimum of \eqref{prob::OCP}.
\\
For the numerical validation, we define three test cases:
\\\\    
(T1) We start with the forcing $\alpha(t) = \xi \sech^2(r (t-t_{peak}))$, for $r,\xi>0$ and $t_{peak}\in\R$. This type of forcing has recently been used to analyse overshoots and tipping mechanisms in a model of the Atlantic Meridional Overturning Circulation (\cite{Lux24}). We consider two cases: 
    \begin{enumerate}[label=(\alph*)]
        \item $\xi=r=1$ and $t_{peak}=0$;
        \item $\xi=5,\ r=4.5$ and $t_{peak}=0$.
    \end{enumerate}
(T2) We use forcing $\alpha = \big(B + A \sin (r t)\big) \indicator_{[t_0,t_1]}$ for $A,B,r>0$ and $t_0<t_1 \in \R$. This is a slight adaptation to \cite{avitabile22}, where a similar forcing was used to study the excitability of neurons in a Network model. Note, the indicator is used to ensure that $\alpha \in L^1$ and $B\geq A$ is necessary for $\alpha \geq 0$. We consider $A=B=1$ and $t_0=0, \ t_1=5$.
\\\\
(T3) Finally, we take a very rigid $\alpha$ to compare both methods for very nonsmooth forcings. In this case $\alpha(t) = \sum_{i=1}^{20} \xi_i \indicator_{t_i,t_{i+1}}(t)$, where $\xi_i$ are randomly sampled from a uniform distribution on $[0,4]$\footnote{The plots are generated for the specific realization of $\xi=$[1.87503898, 2.430023, 1.01158888, 1.93248402, 0.21160768, 3.77938894,
            2.07412297, 0.7612135,  1.53545624, 1.16285869, 0.41309513, 0.63971887, 
            2.60766738, 1.85923124, 3.41572873, 3.71215604, 0.56773364, 2.73073052, 
            1.71919692, 2.52966728].} and $0=t_0<t_1...<t_{20}=5$ are $20$ equidistant time steps between $0$ and $20$. 
\\\\
The results for the cost function values and the execution times for both methods for the different test cases are given in \autoref{tab::1000} for $n=1000$, and in \autoref{tab::5000} for $n=5000$ discretization points. 
The cost function values are computed using the left rectangle rule. Further, for the numerical solution only the time in IPOPT is reported and the overhead of building a CasADi model is neglected. Moreover, the following are single run results performed on an i7-13700H 2400MHz, 16GB RAM laptop. 
\\
\begin{table}[ht!]
 \centering
\begin{tabular}{ |p{3cm}||p{3.5cm}|p{3.5cm}|p{1.9cm}|p{1.9cm}| }
 \hline
External Forcing & Cost Function Num. & Cost Function Aprx. & Exec. Time Num. & Exec. Time Aprx. \\
 \hline
 (T1a) & 0.06 & 0.16 & 1.55s & 0.00022s\\
 (T1b) & 3.08 & 3.19 & 7.59s & 0.00026s\\
 (T2)  & 3.21 & 5.74 & 1.41s & 0.00022s\\
 (T3)  & 11.4 & 16.2 & 2.09s & 0.00024s \\
 \hline
\end{tabular}
    \caption{Table for $n=1000$ discretization points.}
    \label{tab::1000}
\end{table}

\begin{table}[ht!]
 \centering
\begin{tabular}{ |p{3cm}||p{3.5cm}|p{3.5cm}|p{1.9cm}|p{1.9cm}| }
 \hline
External Forcing & Cost Function Num. & Cost Function Aprx. & Exec. Time Num. & Exec. Time Aprx. \\
 \hline
 (T1a) & 0.06 & 0.16 & 7.44s & 0.00023s\\
 (T1b) & 3.09 & 3.13 & 57.6s & 0.00027s\\
 (T2)  & 3.21 & 5.74 & 8.42s & 0.00033s\\
 (T3)  & 11.4 & 16.2 & 11.8s & 0.00040s\\
 \hline
\end{tabular}
    \caption{Table for $n=5000$ discretization points.}
    \label{tab::5000}
\end{table}

\textbf{Execution time:} \\
We see that the approximated solution has a vastly shorter execution time than the numerically computed solution. This is no surprise given the closed formula of the approximated solution \autoref{thm::approximatedcontrol}. In particular, we see that the execution time for the approximated solution does not scale poorly with the number of discretization steps. This is particularly useful in cases where rigid forcings or long time horizons require many discretization steps. For $n=5000$ steps, this can lead to the approximated calculation being $2\times 10^5$ times faster than the numerical calculation (see (T1b)).
\\\\
\textbf{Cost functions values:} \\
The cost functions values of the numerical solution are lower than the corresponding costs for the approximated solution. The differences range from a very large relative error of $1.66$ (see (T1a)) to a small relative error of $10^{-2}$ (see (T1b)). As a consequence of how the approximated rate function \autoref{thm::largekappa} (ii) was derived, by neglecting the recovery of \eqref{eqq::y} to the stable moving equilibrium, it is exact if applied to a delta function. Hence, heuristically speaking, the closer the applied input is to a delta function, the better the approximated control performs. This is exactly what one can observe in (T1a) and (T1b).
\\\\
In the following, we compare both methods with respect to the shape of the controls and the resulting trajectories. 
\autoref{fig::1a}, \autoref{fig::1b}, \autoref{fig::2} and \autoref{fig::3} present the results for all test cases and $n=1000$ discretization points.
In all figures, the red line denotes $-\sqrt{\lambda}$ the minimal value that may be obtained before diverging to $-\infty$. The light blue dotted line corresponds to the approximated control \autoref{thm::approximatedcontrol}, and the dark blue dash-dotted line corresponds to the numerically computed control using CasADi. The solid and dashed black lines in the plot on the bottom right correspond to the stable and unstable equilibria of the autonomous saddle node normal form. Further, $\lambda=0.3$ for all simulations.
\\\\
\textbf{Shape of controls:} \\
Surprisingly, the shapes of the approximated and numerical controls are comparable in all cases, even if the achieved cost function values are very different. The biggest differences are in amplitude and not necessarily general shape (e.g.~(T2)) or smoothness of the control (e.g.~(T2),(T3)).
\\\\
\textbf{Resulting Trajectories:} \\
Finally, the resulting trajectories also have a similar shape. Moreover, the trajectory of the numerical solution is always underneath that of the approximated solution. Furthermore, numerical and approximated trajectories have the same ``spikes'' although often slightly shifted (e.g.~(T2),(T3)).

\begin{figure}[ht!]
    \centering
\includegraphics[width=0.9\linewidth]{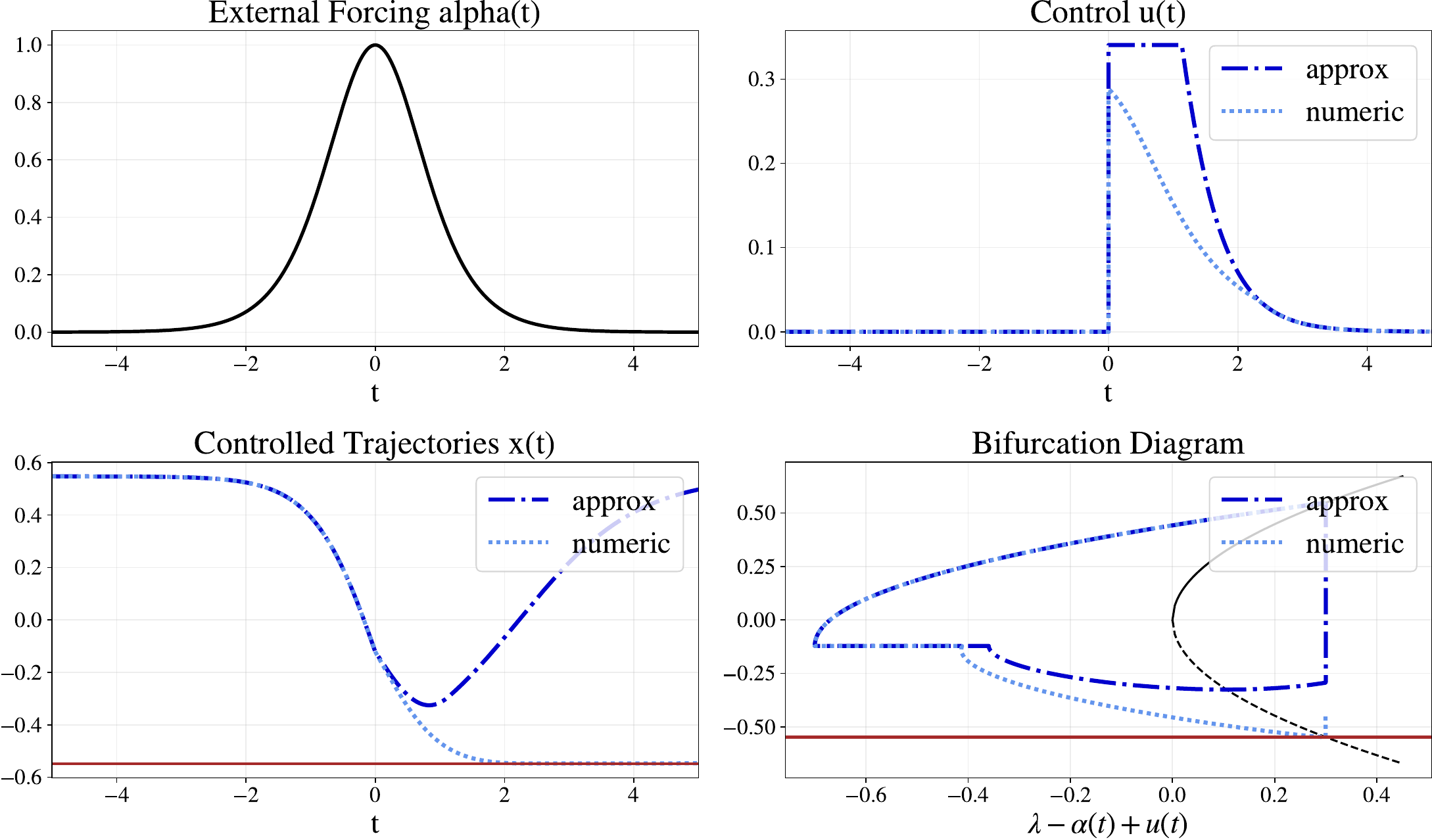}
    \caption{Resulting Trajectories for (T1a).}
    \label{fig::1a}
\end{figure}

\begin{figure}[ht!]
    \centering
    \includegraphics[width=0.9\linewidth]{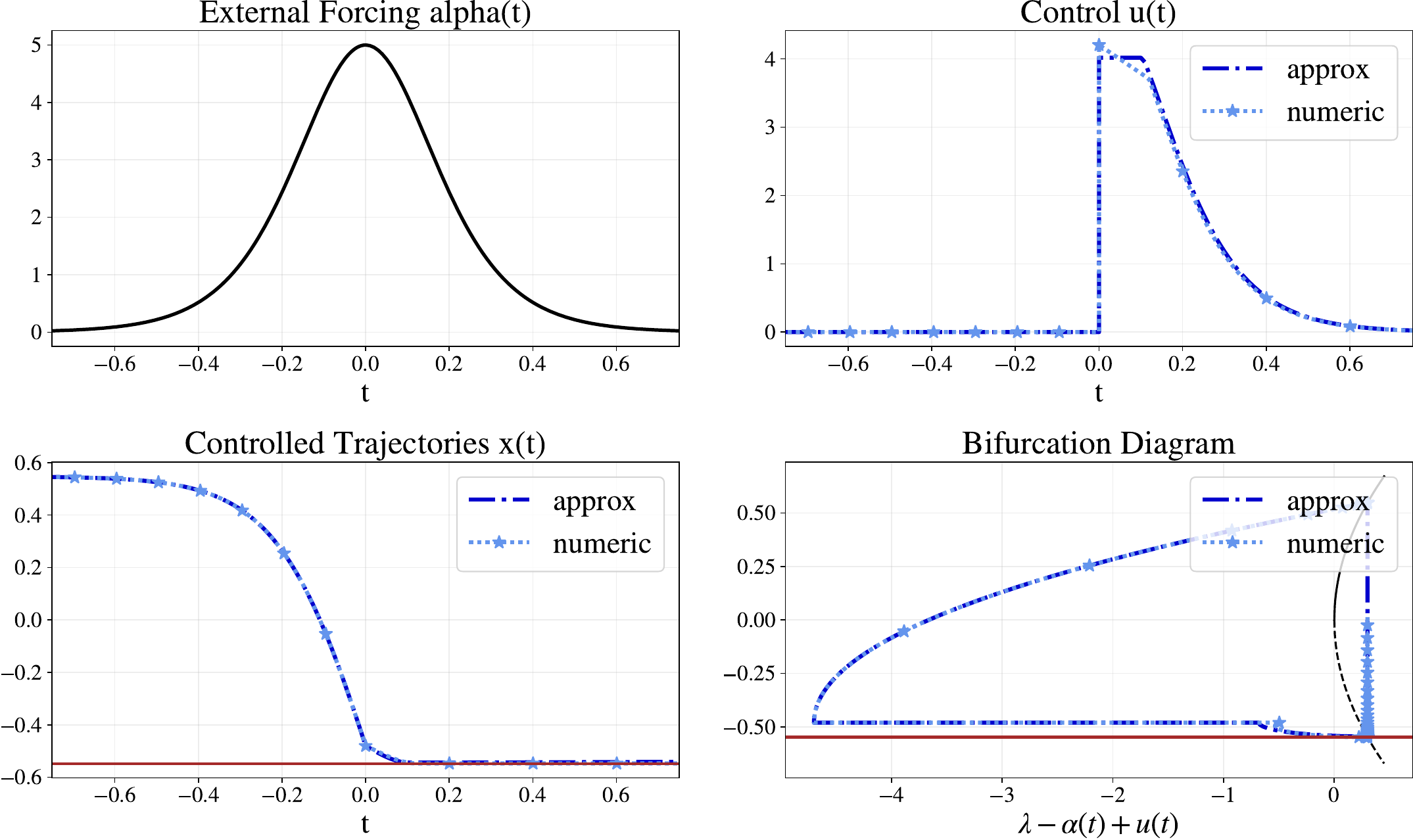}
    \caption{Resulting Trajectories for (T1b).}
    \label{fig::1b}
\end{figure}

\begin{figure}[ht!]
    \centering
    \includegraphics[width=0.9\linewidth]{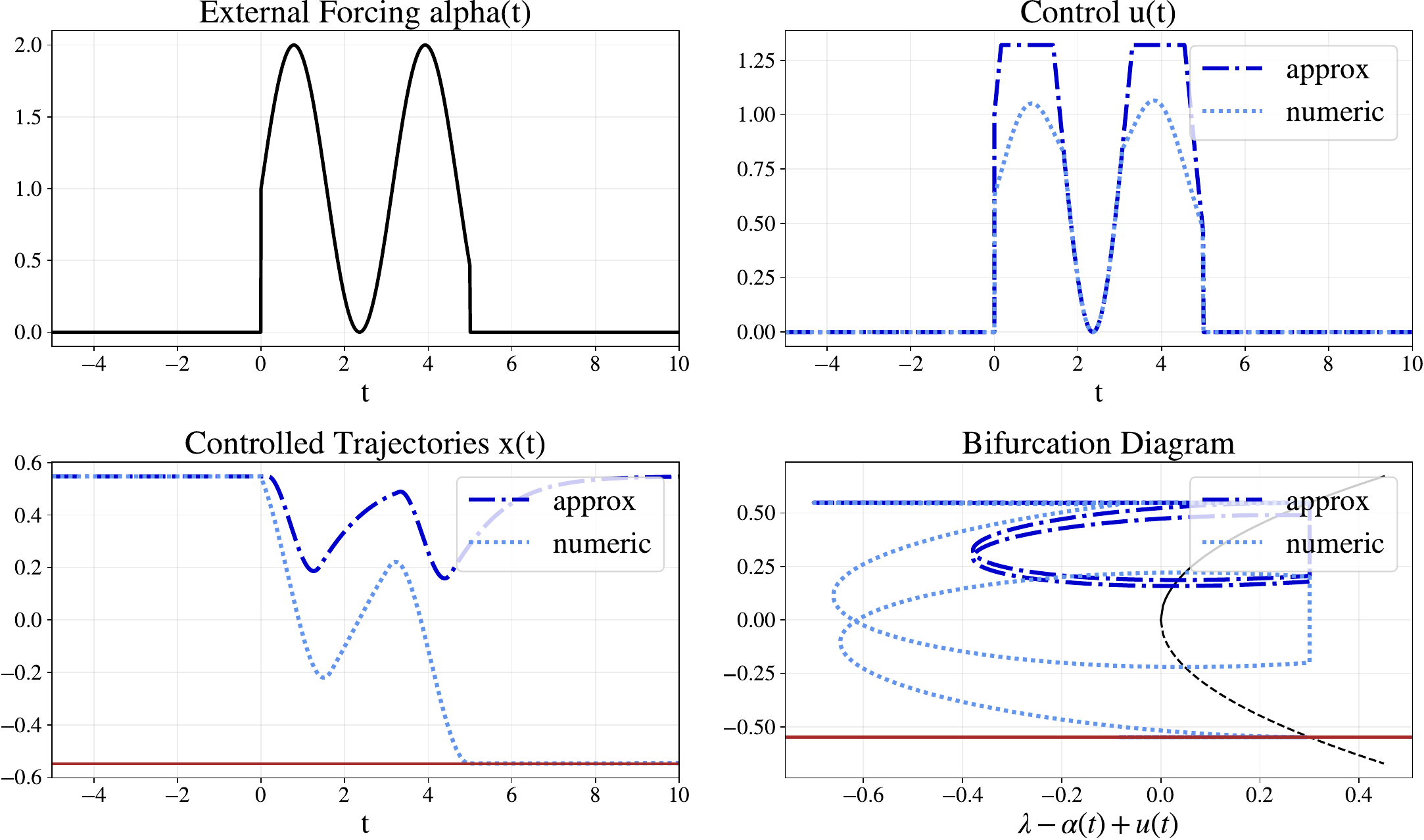}
    \caption{Resulting Trajectories for (T2).}
    \label{fig::2}
\end{figure}

\begin{figure}[ht!]
    \centering
    \includegraphics[width=0.9\linewidth]{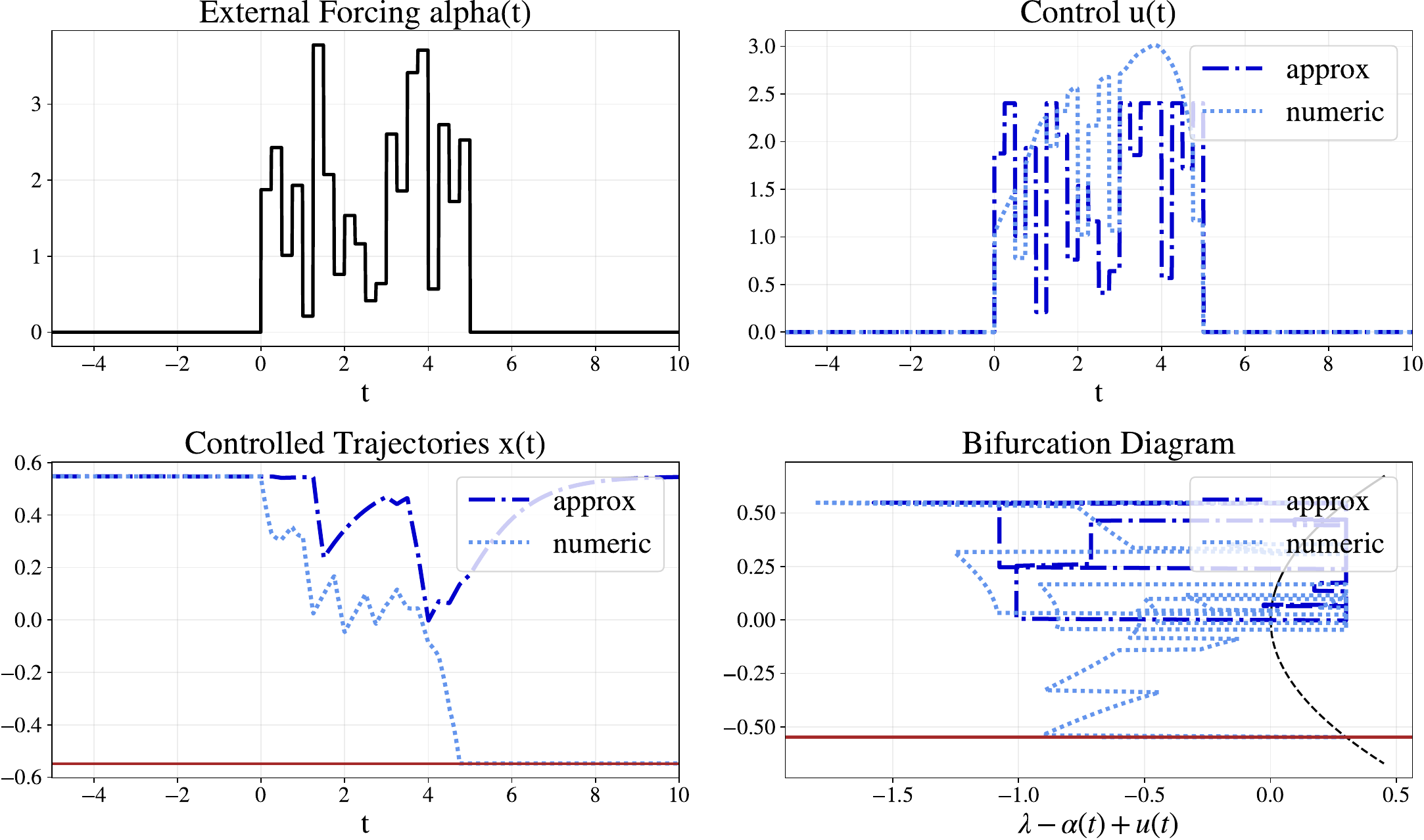}
    \caption{Resulting Trajectories for (T3).}
    \label{fig::3}
\end{figure}

\clearpage

\clearpage

\section{Conclusion and outlook}

We laid the foundation for a dynamically informed optimal control theory of nonautonomous saddle node bifurcations. The key step was the transition from the bifurcation-induced tipping to a rate-induced overshooting problem, where we introduced a rate function $\Kappa$ to assign the critical rate for tipping corresponding to the applied control. We showed that this rate function is indeed well-defined for the nonautonomous saddle node bifurcation normal form by showing that the critical rate is unique. We proved this for initial conditions $x(0)=\mu^x >0$, however, we expect similar arguments to work for $\mu^x > -\sqrt{\lambda}$. 
The introduction of this rate function led to the reformulation of our constrained optimal control problem \eqref{prob::OCP} into an unconstrained minimization problem \eqref{prob::RCP}. This enables the use of powerful optimization methods from the toolbox of unconstrained nonlinear optimization, which can be further exploited in future research. Additionally, we used an approximation for the rate function approximation, based on which we derived a closed form expression for the optimal control, which is still guaranteed to respect the non-tipping requirement. This opens up the floor to connect to the fields of Uncertainty Quantification (uncertain parameters in dynamics and/or noise), data assimilation and model predictive control, where the optimal control problem \eqref{prob::OCP} needs to be solved many times. 
\\
Moreover, our reformulation naturally provides flexibility between computational efficiency and precision, without dropping the non-tipping constraint, by choosing the accuracy at which the rate function is approximated. This was further explored in Section \ref{sec::rateApproxTheory} and numerically validated in Section \ref{sec::numerics}.
\\
\\
Note that the reformulation of \eqref{prob::OCP} using the critical rate function $\Kappa$ is not the only choice. Alternatively, one can also define a critical size function $\xi$ and consider
\begin{align*}
    \min_{g\in\SET(\R^+)} \int_0^\infty (\alpha(t) - \xi(g)g(t))^2 \dd t.
\end{align*}
Then, the argumentation in Section \ref{sec::reformOCproblem} would still hold if one can show that there exists a unique critical size for every forcing $g\in\SET(\R^+)$. In fact, one might even consider any combination of rate-, size- and phase-induced transitions. The reason we decided for the rate-induced approach here is that it seems the least invasive, as it only changes the time scale at which the effective forcing $g$ is applied, but not $g$ itself. However, if one aims to solve the reformulated problem using variational approaches, then, the size function may be more suitable as it yields a more tractable and more explicit functional derivative of the loss function.
\\
\\
Naturally, an important next step is to extend the theory developed here to more complex multidimensional models, such as the Hodgkin–Huxley model \cite{Wilson22} or the neuronal network model \cite{avitabile22}. One important advantage of our newly developed method compared to more general approaches for nonlinear constrained optimization problems is that the method really exploits the inherent bifurcation structure of the given problem, which suggests better scaling properties in the multi-dimensional setting for saddle node bifurcations. This is because the core bifurcating dynamics remain present, see \cite{universal_explosive_phenomena}. To benefit therefrom, the framework of time-dependent center manifold reduction (\cite[Sec.~4]{NonautonomousDynamicalSystems},\cite{wiggins2025}) seems well suited and is a very interesting starting point to extend our theory to more complex models with inherent saddle node bifurcation structure and use our technique in application scenarios.

\clearpage

\section{Appendix}\label{sec::appendix}

\begin{proposition}\label{app::prop::cond_allsystems}
    Let $p \in \SET(\R)$ and $\lambda, \kappa>0$, then:
    \begin{align*}
        & f(t,x) = \lambda - p(t) -x^2; \\
        & g(t,x) = \lambda - \left(x - \int_{-\infty}^t p(s) \dd s\right)^2; \\
        & q(t,x) = \frac{1}{\kappa^2} \cos^2(x) + (p(t) - \lambda) \sin^2(x);
    \end{align*}
    all satisfy \autoref{cond::1}. In addition, $f,g$ are strictly concave and coercive, which means that
    \begin{itemize}
        \item Coercivity: There exists a subset $\mathcal{R}\subset \R$ of full Lebesgue measure and $\delta>0$ such that 
        \begin{align*}
        \limsup_{x \to \pm \infty} f(t,x)/x^2 < -\delta  \qquad \limsup_{x \to \pm \infty} g(t,x)/x^2 < -\delta.
        \end{align*}
        uniformly for all $t\in \mathcal{R}$
        \item Strict Concavity: For all $j\in\N$ there exists a constant $\delta_j>0$ such that $l$-a.a. 
        \begin{align*}
            \sup_{x_1,x_2 \in[-j,j], x_1\neq x_2} \frac{f_x(t,x_2) - f_x(t,x_1)}{x_2-x_1} < -\delta_j \qquad \sup_{x_1,x_2 \in[-j,j], x_1\neq x_2} \frac{g_x(t,x_2) - g_x(t,x_1)}{x_2-x_1} < -\delta_j.
        \end{align*}
    \end{itemize}
\end{proposition}

\begin{proof}
    We start by proving (i)-(iv) of \autoref{cond::1}. \\
    (i) \& (iii) Clearly, all three functions are Borel measurable and differentiable in $x$. \\\\
    (ii) Note that
    \begin{align*}
        & \sup_{x\in[-j,j]} f(t,x) \leq ||p||_{\infty;\R} + \lambda  \qquad  &&\text{and} \qquad\sup_{x\in[-j,j]} -f(t,x) \leq ||p||_{\infty;\R}+ j^2; \\
        & \sup_{x\in[-j,j]} g(t,x) \leq \lambda  \qquad && \text{and} \qquad \sup_{x\in[-j,j]} -g(t,x) \leq ||p||_{\infty;\R}^2 + j^2 + 2j||p||_{\infty;\R}; \\
        & \sup_{x\in[-j,j]} |q(t,x)| \leq \frac{1}{\kappa^2} + ||p||_{\infty;\R} + \lambda .
    \end{align*}
    \\
    (iv) We treat the functions separately. The statements for $f$ are trivial, since $f(t,x_1) - f(t,x_2) = x_2^2 -x_1^2$ and $f_x(t,x_1) - f_x(t,x_2) = 2(x_2 -x_1)$. For $g$, we get 
    \begin{align*}
        \sup_{x_1,x_2\in[-j,j]} \frac{|g(t,x_2) - g(t,x_1)|}{|x_2-x_1|} 
        &= \sup_{x_1,x_2\in[-j,j]} \frac{|x_1^2 -x_2^2 - 2(x_1-x_2) \int_0^t p(s) \dd s|}{|x_2-x_1|} \\
        &\leq \sup_{x_1,x_2\in[-j,j]} \frac{|x_1^2 -x_2^2|}{|x_2-x_1|} + 2 ||p||_{1;\R} \sup_{x_1,x_2\in[-j,j]} \frac{|x_1-x_2|}{|x_2-x_1|}
    \end{align*}
    from which the first claim follows. Further, 
    \begin{align*}
        \sup_{x_1,x_2\in[-j,j]} \frac{|g_x(t,x_2) - g_x(t,x_1)|}{|x_2-x_1|} 
        &= \sup_{x_1,x_2\in[-j,j]} \frac{|2(x_1-x_2)|}{|x_2-x_1|} 
    \end{align*}
    which shows the second claim. \\
    Finally, for $q$ we obtain
    \begin{align*}
        &\sup_{x_1,x_2\in[-j,j]} \frac{|q(t,x_2) - q(t,x_1)|}{|x_2-x_1|} 
        = \sup_{x_1,x_2\in[-j,j]} \frac{|\frac{1}{\kappa^2}(\cos^2(x_2) - \cos^2(x_1)) + (p(t) - \lambda) (\sin^2(x_2) -\sin^2(x_1)) |}{|x_2-x_1|} \\
        &\leq \frac{1}{\kappa^2}\sup_{x_1,x_2\in[-j,j]} \frac{|\cos^2(x_2) - \cos^2(x_1) |}{|x_2-x_1|} + (||p||_{\infty;\R} + \lambda)\sup_{x_1,x_2\in[-j,j]} \frac{|\sin^2(x_2) -\sin^2(x_1) |}{|x_2-x_1|},
    \end{align*}
    which by the differentiability of $\sin$ and $\cos$ can naturally be upper bounded for any $j\in\N$. Moreover, 
    \begin{align*}
        \sup_{x_1,x_2\in[-j,j]} \frac{|q_x(t,x_2) - q_x(t,x_1)|}{|x_2-x_1|} 
        &= \sup_{x_1,x_2\in[-j,j]} \frac{|2(-1/\kappa^2 + p(t) - \lambda)(\sin(x_2)\cos(x_2) - \sin(x_1)\cos(x_1))|}{|x_2-x_1|} \\
        &\leq 2(1/\kappa^2 + ||p||_{\infty;\R} + \lambda) \sup_{x_1,x_2\in[-j,j]}\frac{|\sin(x_2)\cos(x_2) - \sin(x_1)\cos(x_1)|}{|x_2-x_1|},
    \end{align*}
    where again the smoothness of $\sin$ and $\cos$ implies that one can upper bound the last term for any $j\in \N$. 
    \\\\
    The coercivity of $f,g$ is immediately apparent, since $p\in L^1(\R)$. 
    \\
    For concavity, note that
    \begin{align*}
        \frac{f_x(t,x_2) - f_x(t,x_1)}{x_2-x_1} = -2\frac{x_2 - x_1}{x_2-x_1} = -2
    \end{align*}
    and 
    \begin{align*}
        \frac{g_x(t,x_2) - g_x(t,x_1)}{x_2-x_1} = -2\frac{x_2 - x_1}{x_2-x_1} = -2.
    \end{align*}
\end{proof}

\vspace{0.75cm}

\begin{proof}[Proof of \autoref{thm::largekappa}]
    We argue via the $y$-system \eqref{eqq::y}. The statements for \eqref{eqq::x} follow by \autoref{pro::xysamecase}. 
    \\\\
    (ii) Let $\attr^x(0) \in (-\sqrt{\lambda}, \sqrt{\lambda})$ and $g\geq 0$. Note that $ \attr^y(0) = \attr^x(0) + \int_{-\infty}^0 \alpha(t) \dd t$ and let $\kappa > \frac{||g||_{1;[0,\infty)}}{\sqrt{\lambda} + \attr^x(0)}$. 
    Define
    \begin{align*}
        \varepsilon \coloneqq \min\left\{ \sqrt{\lambda}, \  \sqrt{\lambda} + \attr^x(0) - \frac{1}{\kappa} ||g||_{1;[0,\infty)} \right\} \quad \in (0,\sqrt{\lambda}),
       \end{align*}
       then
       \begin{align}\label{eqq::proof::normgestimate}
        \frac{1}{\kappa} ||g||_{1;[0,\infty)} \leq \sqrt{\lambda} + \attr^x(0) - \varepsilon
    \end{align}
    and therefore 
    \begin{align}\label{eqq::proof::initialcondbigger}
        \begin{split}
        &-\sqrt{\lambda} + \int_{-\infty}^0 \alpha(t) \dd t + \frac{1}{\kappa} \int_0^{\infty} g(t) \dd t + \frac{\varepsilon}{2} 
        \leq -\sqrt{\lambda} + \int_{-\infty}^0 \alpha(t) \dd t + \frac{1}{\kappa} ||g||_{1,[0,\infty)} + \frac{\varepsilon}{2} \\
        &\leq \attr^x(0) + \int_{-\infty}^0 \alpha(t) \dd t - \frac{\varepsilon}{2} < \attr^y(0).
        \end{split}
    \end{align}
    Assume that \eqref{eqq::y}$_\kappa$ is in Case B or C, by \autoref{thm::casesy} this means that either $\lim_{t\to\infty} \attr^y(t) = -\sqrt{\lambda} + \int_{-\infty}^0 \alpha(t) \dd t + \frac{1}{\kappa} \int_0^\infty g(t) \dd t$ or $\attr^y$ diverges to $-\infty$ in finite time. In both cases the continuity of $\attr^y$ combined with \eqref{eqq::proof::initialcondbigger} implies the existence of $t_1, \delta > 0$ such that
    \begin{align}
        \nonumber &\attr^y(t_1) = -\sqrt{\lambda} + \int_{-\infty}^0 \alpha(t) \dd t + \frac{1}{\kappa} \int_0^{\infty} g(t) \dd t + \frac{\varepsilon}{2} \qquad \text{and}\\
         \label{eqq::proof::smallerinterval}&\attr^y(t) < -\sqrt{\lambda} + \int_{-\infty}^0 \alpha(s) \dd s + \frac{1}{\kappa} \int_0^{t} g(s) \dd s + \frac{\varepsilon}{2} \qquad \forall \  t\in(t_1, t_1+\delta).
     \end{align} 
    However,
     \begin{align*}
         \dot{\attr}^y(t_1) &= \lambda + \left( \attr^y(t_1) - \int_{-\infty}^0 \alpha(t) \dd t + \int_0^{t_1} g(t \kappa) \dd t\right)^2 \\
         &= \lambda + \left( -\sqrt{\lambda} + \int_{-\infty}^0 \alpha(t) \dd t + \int_0^{\infty} g(t\kappa) \dd t + \frac{\varepsilon}{2}  - \int_{-\infty}^0 \alpha(t) \dd t + \int_0^{t_1} g(t \kappa) \dd t\right)^2 \\
         &= \lambda + \left( \int_{t_1}^\infty g(t\kappa) \dd t - \sqrt{\lambda} + \frac{\varepsilon}{2}\right)^2 
     \end{align*}
     and therefore $g\geq 0$ together with \eqref{eqq::proof::normgestimate} imply 
     \begin{align*}
         \dot{\attr}^y(t_1) = \lambda + \bigg(\underbrace{ \int_{t_1}^\infty g(t\kappa) \dd t  -\sqrt{\lambda} + \frac{\varepsilon}{2} }_{\in (-\sqrt{\lambda}, \attr^x(0) - \varepsilon/2) \subset (-\sqrt{\lambda},\sqrt{\lambda})}\bigg)^2 >0
     \end{align*}
     contradicting \eqref{eqq::proof::smallerinterval}. \\\\
     (i) We, again, start with the case $\attr^x(0) \in (-\sqrt{\lambda}, \sqrt{\lambda})$. The structure of the proof is similar to (ii), however, since $g$ can be negative now, we could have $\int_{t_1}^\infty g(t\kappa) \dd t - \sqrt{\lambda} + \varepsilon/2 \leq-\sqrt{\lambda}$. This needs to be circumvented. Let 
     \begin{align*}
         \varepsilon \coloneqq \min\left\{\sqrt{\lambda} + \attr^x(0) - \frac{1}{\kappa} ||g||_{1;[0,\infty)}, \sqrt{\lambda} - \attr^x(0) - \frac{1}{\kappa} ||g||_{1;[0,\infty)} \right\}
     \end{align*}
     and $\kappa_2$ large enough such that 
     \begin{align*}
         & \sqrt{\lambda} + \attr^x(0) - \frac{3}{\kappa} ||g||_{1;[0,\infty)} \in (0,2\sqrt{\lambda})\\
         &\sqrt{\lambda} - \attr^x(0) - \frac{3}{\kappa} ||g||_{1;[0,\infty)} \in (0,2\sqrt{\lambda})
     \end{align*}
     for all $\kappa>0$. Note that this also implies $\varepsilon > 0$. By the same steps as before, we arrive at \eqref{eqq::proof::smallerinterval}, where again
     \begin{align*}
         \dot{\attr}^y(t_1) = \lambda + \left( \int_{t_1}^\infty g(t\kappa) \dd t -\sqrt{\lambda} + \frac{\varepsilon}{2}\right)^2 .
     \end{align*}
     Further, for $\attr^x(0)\in [0,\sqrt{\lambda})$
     \begin{align*}
        1.\qquad &\int_{t_1}^\infty g(t\kappa) \dd t -\sqrt{\lambda} + \frac{\varepsilon}{2} 
         \geq -\frac{1}{\kappa} ||g||_{1;[0,\infty)} -\sqrt{\lambda} + \frac{\varepsilon}{2} \\
         &= -\frac{1}{\kappa} ||g||_{1;[0,\infty)} -\sqrt{\lambda} + \frac{1}{2} \left(\sqrt{\lambda} - \attr^x(0) -\frac{1}{\kappa} ||g||_{1;[0,\infty)} \right) \\
         &= \frac{1}{2} \left( -\sqrt{\lambda} - \attr^x(0) - \frac{3}{\kappa} ||g||_{1;[0,\infty)}  \right) > -\sqrt{\lambda}
         \\\\
        2.\qquad &\int_{t_1}^\infty g(t\kappa) \dd t -\sqrt{\lambda} +\frac{\varepsilon}{2} 
         \leq \frac{1}{\kappa} ||g||_{1;[0,\infty)} -\sqrt{\lambda} + \frac{\varepsilon}{2} \\
         &= \frac{1}{\kappa} ||g||_{1;[0,\infty)} -\sqrt{\lambda} + \frac{1}{2} \left(\sqrt{\lambda} - \attr^x(0) -\frac{1}{\kappa} ||g||_{1;[0,\infty)} \right) \\
         &= \frac{1}{2} \left( -\sqrt{\lambda} - \attr^x(0) + \frac{1}{\kappa} ||g||_{1;[0,\infty)}  \right) < \sqrt{\lambda}
     \end{align*}
     Similarly, for $\attr^x(0) \in (-\sqrt{\lambda},0)$:
     \begin{align*}
        1.\qquad &\int_{t_1}^\infty g(t\kappa) \dd t -\sqrt{\lambda} + \frac{\varepsilon}{2} 
         \geq -\frac{1}{\kappa} ||g||_{1;[0,\infty)} -\sqrt{\lambda} + \frac{\varepsilon}{2} \\
         &= -\frac{1}{\kappa} ||g||_{1;[0,\infty)} -\sqrt{\lambda} + \frac{1}{2} \left(\sqrt{\lambda} + \attr^x(0) -\frac{1}{\kappa} ||g||_{1;[0,\infty)} \right) \\
         &= \frac{1}{2} \left( -\sqrt{\lambda} + \attr^x(0) - \frac{3}{\kappa} ||g||_{1;[0,\infty)}  \right) > -\sqrt{\lambda};
         \\\\
        2.\qquad &\int_{t_1}^\infty g(t\kappa) \dd t -\sqrt{\lambda} +\frac{\varepsilon}{2} 
         \leq \frac{1}{\kappa} ||g||_{1;[0,\infty)} -\sqrt{\lambda} + \frac{\varepsilon}{2} \\
         &= \frac{1}{\kappa} ||g||_{1;[0,\infty)} -\sqrt{\lambda} + \frac{1}{2} \left(\sqrt{\lambda} + \attr^x(0) -\frac{1}{\kappa} ||g||_{1;[0,\infty)} \right) \\
         &= \frac{1}{2} \left( -\sqrt{\lambda} + \attr^x(0) + \frac{1}{\kappa} ||g||_{1;[0,\infty)}  \right) < \sqrt{\lambda}.
     \end{align*}
     In any case, we have 
     \begin{align*}
         \dot{\attr}^y(t_1) = \lambda + \bigg( \underbrace{\int_{t_1}^\infty g(t\kappa) \dd t -\sqrt{\lambda} + \frac{\varepsilon}{2}}_{\in (-\sqrt{\lambda}, \sqrt{\lambda})}\bigg)^2 > 0 
     \end{align*}
     contradicting \eqref{eqq::proof::smallerinterval}.
     \\\\
     Consider the solution $t\mapsto y_\kappa(t,0,0)$ of \eqref{eqq::y}$_\kappa$ with $y(0)=\int_{-\infty}^0 \alpha(t) \dd t$ and $\kappa$ large enough such that the above holds. By the same argumentation as before and \autoref{thm::casesy}
     \begin{align*}
         y_\kappa(t) \to \sqrt{\lambda} + \int_{-\infty}^0 \alpha(t) \dd t + \int_0^\infty g(t\kappa) \dd t .
     \end{align*}
     In particular, there exists $\Bar{t}$ such that 
     \begin{align*}
        y_\kappa(t) \geq \int_{-\infty}^0 \alpha(t) \dd t + \int_0^\infty g(t\kappa) \dd t  \qquad \forall \  t \geq \Bar{t}.
     \end{align*}
     By the comparison principle Appendix \autoref{thm::comparison}, we have $\attr^y_\kappa(t) \geq y_\kappa(t,0,0)$ for all $t\geq 0$. It follows that 
     \begin{align*}
         \lim_{t\to\infty} \attr^y_\kappa(t) \geq \int_{-\infty}^0 \alpha(s) \dd s + \int_0^\infty g(s\kappa) \dd s .
     \end{align*}
     By \autoref{thm::casesy} $\lim_{t\to\infty} \attr^x_\kappa(t) = \sqrt{\lambda} +\int_\R h(t;\kappa) \dd t$ and \eqref{eqq::y}$_\kappa$ is in Case A. \\\\
     (iii) Assume that \eqref{eqq::y}$_{\kappa_2}$ is in Case C. Since \eqref{eqq::y}$_{\kappa_2 + 1}$ is in Case A by (i), \autoref{thm::casesopen} (iii) shows that there is a $\kappa \in (\kappa_2,\kappa_2+1)$ such that \eqref{eqq::y}$_{\kappa}$ is in Case B. This is a contradiction to the properties of $\kappa_2$ shown in (i).
\end{proof}

\begin{proposition}\label{pro::smallKappa_underCond}
    Let $\lambda>0$ and $\alpha\in\SET(\R)$ such that $\attr^x(0)$ exists. Further, let $g\in\SET(\R^+)$ such that $l(\{g>\lambda\}) > 0$. Then, there exists $\kappa_1 >0$ such that \eqref{eqq::x}$_{\kappa_1}$ and \eqref{eqq::y}$_{\kappa_1}$ are in Case C.
\end{proposition}

\begin{proof}
We argue by \eqref{eqq::x}, the statement for \eqref{eqq::y} follows by \autoref{pro::xysamecase}. 
\\
We first show that if there exists $t\geq 0$ such that $\attr^x(t) < - \sqrt{\lambda + ||g||_{\infty;[0,\infty)}}$, then \eqref{eqq::x} is in Case C. Assume such $t$ exists and let $\delta \coloneqq (- \sqrt{\lambda + ||g||_{\infty;[0,\infty)}} - \attr^x)/2 >0$. Then
\begin{align*}
    \dot{\attr}^x(t) = \lambda - g(t\kappa) - x^2(t)  < \lambda - g(t\kappa) - (\lambda + ||g||_{\infty;[0,\infty)}) - \delta^2 \leq - \delta^2.
 \end{align*}
 Hence,
 \begin{align*}
     \attr^x(s) \leq \attr^x(t) - \delta^2(s-t) \longrightarrow - \infty 
 \end{align*}
as $(s-t) \to \infty$. Since, by \autoref{thm::casesx}, $\attr^x$ is unbounded if and only if it is in Case C this shows the first claim. 
\\
\\
Next, we show that $\attr^x(t) \leq \sqrt{\lambda + ||\alpha||_{\infty;\R} + ||g||_{\infty;[0,\infty)} + 1}$ for all $t\in \R$. To see that assume that there is $t\in\R$ such that $\attr^x = \sqrt{\lambda + ||\alpha||_{\infty;\R} + ||g||_{\infty;[0,\infty)} + 1}$, then
\begin{align*}
    \dot{\attr}^x(t) \leq \lambda + (||\alpha||_{\infty;\R} + ||g||_{\infty;[0,\infty)}) - \sqrt{\lambda + ||\alpha||_{\infty;\R} + ||g||_{\infty;[0,\infty)} + 1}^2 < 0.
\end{align*}
This shows the second claim.
\\
\\
We now show the statement of \autoref{pro::smallKappa_underCond}.  
Assume $l(\{g>\lambda\})>0$, then there is a $\delta>0$ such that $l(\{g>\lambda+\delta\})>0$. To see this, assume that for all $\delta >0$, we have $l(\{g>\lambda+\delta\})=0$. Then,
$$
l(\{g>\lambda\})>0 = l\left(\bigcup_{n\in\N} \left\{g>\lambda + \frac{1}{n}\right\}\right) \leq \sum_{n\in\N} l(\{g>\lambda+\frac{1}{n}\})=0
$$
gives a contradiction. \\
Hence, there is $\delta>0$ and a set $\mathcal{A}$ of positive measure such that $g > \lambda + \delta$ on $\mathcal{A}$. By the Lebesgue density theorem \cite[Thm.~21.29]{secondCourseOnRealFunctions}, there exists $\Bar{t} \in \mathcal{A}$ such that
$$
\frac{l(\mathcal{A} \cap B_\varepsilon(\Bar{t}))}{l(B_\varepsilon(\Bar{t}) } \overset{\varepsilon\to 0}{\longrightarrow} 1.
$$
In particular, $\forall \  \nu>0$ there is an $\varepsilon>0$ such that 
$$
\frac{l(\mathcal{A} \cap B_\varepsilon(\Bar{t}))}{l(B_\varepsilon(\Bar{t})) } \geq 1 - \nu \Longleftrightarrow l(\mathcal{A} \cap B_\varepsilon(\Bar{t})) \geq 2 \varepsilon(1-\nu).
$$
Define $A \coloneqq \mathcal{A}\cap B_\varepsilon (\Bar{t})$, $B \coloneqq B_\varepsilon(\Bar{t}) \backslash \mathcal{A}$ and note that $l(B) \leq 2\varepsilon\nu$. Similarly, by the scaling properties of Lebesgue measure, we obtain for $\kappa>0$ and by considering $g(\cdot \kappa)$ instead of $g(\cdot)$ that
\begin{align*}
    & l(A^\kappa) = l(\mathcal{A}^\kappa \cap B_{\varepsilon/\kappa}(\Bar{t})) \geq 2\frac{\varepsilon}{\kappa} (1-\nu) \\
    & l(B^\kappa) \leq 2\frac{\varepsilon}{\kappa}\nu.
\end{align*}
Further,
\begin{align*}
    & \dot{x}_\kappa = \lambda - g(t\kappa) - x^2_\kappa \leq -\delta \qquad &&\text{on } A^\kappa \\
    & \dot{x}_\kappa = \lambda - g(t\kappa) - x^2_\kappa \leq \lambda + ||g||_{\infty;[0,\infty)} \qquad &&\text{on } B^\kappa. 
\end{align*}
Hence
\begin{align*}
    &\attr^x_\kappa(t+\frac{\varepsilon}{\kappa}) = x(\Bar{t} + \frac{\varepsilon}{\kappa}, \Bar{t}-\frac{\varepsilon}{\kappa}, a(\Bar{t} - \frac{\varepsilon}{\kappa})) 
    \leq x(\Bar{t} + \frac{\varepsilon}{\kappa}, \Bar{t}-\frac{\varepsilon}{\kappa}, \sqrt{\lambda + ||\alpha||_{\infty;\R} + ||g||_{\infty;[0,\infty)} + 1}) \\ 
    &\leq \sqrt{\lambda + ||\alpha||_{\infty;\R} + ||g||_{\infty;[0,\infty)} + 1} + -\delta \frac{2\varepsilon}{\kappa} (1-\nu) + (\lambda + ||g||_{\infty;[0,\infty)})\frac{2\varepsilon}{\kappa}\nu  \\
    &= \sqrt{\lambda + ||\alpha||_{\infty;\R} + ||g||_{\infty;[0,\infty)} + 1} + 2\frac{\varepsilon}{\kappa} \big[(\lambda + ||g||_{\infty;[0,\infty)} + \delta) \nu - \delta \big],
\end{align*}
where upon choosing $\nu \coloneqq \frac{\delta}{2(\lambda + ||g||_{\infty ;[0\infty)}+ \delta)}> 0$, the inner bracket becomes negative and therefore for small enough $\kappa>0$, we obtain
$$
x(\Bar{t} + \varepsilon, \Bar{t}-\varepsilon, a(\Bar{t} - \varepsilon)) < -\sqrt{\lambda + ||g||_{\infty;[0,\infty)}},
$$
i.e.~$\attr^x$ diverges to $-\infty$ and therefore, $\eqref{eqq::x}$ is in Case C.
\end{proof}

The following is a straightforward application of \cite[Thm.~2.11]{PiecewiseUniformlyContinuous24} 
\begin{theorem}\label{thm::lambdastar}
    Let $q,p$ be bounded and piecewise uniformly continuous and consider the equation
    \begin{align}\label{eqq::thmlambda}
        \dot{x} = \lambda + p(t) + q(t) x - x^2.
    \end{align}
    Then, there exists a unique $\lambda^*=\lambda^*(q,p): C_{bu}(\R) \times C_{bu} (\R) \to \R$ such that:
    \begin{itemize}
        \item If $\lambda > \lambda^*$, \eqref{eqq::thmlambda} is in Case A;
        \item If $\lambda = \lambda^*$, \eqref{eqq::thmlambda} is in Case B;
        \item If $\lambda < \lambda^*$, \eqref{eqq::thmlambda} is in Case C.
    \end{itemize}

Moreover, let $\Bar{q},\Bar{p}$ be bounded and uniformly continuous, then there exists a constant $m$ such that
\begin{align*}
    | \lambda^*(q,p) - \lambda^*(\Bar{q}, \Bar{p})| \leq m (||\Bar{q} - q||_{\infty;\R} + ||\Bar{p}-p||_{\infty;\R}).
\end{align*}
In particular, $\lambda^*$ is continuous in the $L^\infty$ topology.
\end{theorem}

\autoref{thm::lambdastar} applies in particular to the $y$-system \eqref{eqq::y} with $q=2H_\kappa$ and $p=-H_\kappa^2$, since
\begin{align*}
        \Dot{y}(t) = \lambda - \left( y(t) - H_\kappa(t)) \dd s\right)^2 = \lambda + 2 H_\kappa(t) - H_\kappa(t)^2 - y^2
\end{align*}
and $H_\kappa$ is bounded and uniformly continuous for all $\kappa>0$.

\vspace{0.75cm}

\begin{proof}[Proof of \autoref{lem::propH}] \label{app::proof:propH}
(i) Boundedness readily follows from $h(t;\kappa) \in L^1$ for any $\kappa>0$. Further, uniform continuity follows from the fact that $H_\kappa$ (resp $H_\kappa^2$) is differentiable with
\begin{align*}
    & |H_\kappa'(t)| = \left| \bigg( \int_{-\infty}^t h(t;\kappa) \bigg)' \right| = |h(t;\kappa)| \leq ||h||_{\infty;\R} \\
    & |H_\kappa'(t)| = \left| 2\bigg(\int_{-\infty}^t h(t;\kappa) \bigg) h(t;\kappa) \right| \leq  2||h||_1 ||h||_{\infty;\R}.
\end{align*}
Hence, $H_\kappa$ (resp. $H_\kappa^2$) are differentiable on $\R$ with bounded derivative, therefore uniformly continuous. 
\\\\
(ii) To shorten the notation, we write $H_n$ ($H_{n_k}$) for $H_{\kappa_n}$ ($H_{\kappa_{n_k}}$) and $H$ for $H_\kappa$. First note that the sequence $(H_n)_{n\in\N}$ is uniformly bounded and equicontinuous since:
\begin{align*}
|H_n(t)| \leq ||\alpha||_{1;\R} + \frac{1}{\kappa_n} ||g||_{1;[0,\infty)} \leq ||\alpha||_{1;\R} +\frac{1}{a} ||h||_{1;\R} 
\end{align*}
and
\begin{align*}   
|H_n(t) - H_n(s)| = \left\{\begin{array}{cccc}
    &\left| \int_s^t \alpha(l) \dd l \right| \leq |t-s|||\alpha||_{\infty;\R} &&  s,t <0\\ 
    &\left|\frac{1}{\kappa_n} \int_{s\kappa_n}^{t\kappa_n} g(l) \dd l \right| \leq |t-s| ||g||_{\infty;[0,\infty)} && s,t \geq 0 \\
    &\left| \int_s^0 \alpha(l) \dd l + \frac{1}{\kappa_n} \int_{0}^{t\kappa_n} g(l) \dd l \right| \leq |t-s| \max({||\alpha||_{\infty;\R}, ||g||_{\infty;[0,\infty)}}) && s<0<t \\
     &\left| \int_t^0 \alpha(l) \dd l + \frac{1}{\kappa_n} \int_{0}^{s\kappa_n} g(l) \dd l \right| \leq |t-s| \max({||\alpha||_{\infty;\R}, ||g||_{\infty;[0,\infty)}}) && t<0<s.
     \end{array}\right.
\end{align*}
Now let $\varepsilon>0$, we want to show that there exists $k_0$ and a subsequence $n_k$ such that for all $k\geq k_0$ 
$$
|H_{n_k}(t) -H(t)| = \left| \int_{-\infty}^{t} h(s;\kappa) \dd s - \int_{-\infty}^{t}h(s;\kappa_{n_k}) \dd s\right|< \varepsilon \qquad \forall \  t\in\R.
$$
For $t\in(-\infty,0)$, the assertion is trivial, since
$$
    |H_{n}(t) -H(t)| = \left|\int_{-\infty}^t \alpha(s) \dd s - \int_{-\infty}^t \alpha(s) \dd s  \right| = 0.
$$
Hence, in the following, we consider $t\geq 0$ and therefore 
$$
    |H_{n}(t) -H(t)| = \left| \frac{1}{\kappa_{n}}\int_{0}^{t\kappa_{n}} g(s) \dd s - \frac{1}{\kappa}\int_{0}^{t\kappa} g(s) \dd s  \right|.
$$

Since $g\in L^1$ and $\kappa_n,\kappa \subset [a,b]$ there exists $T>0$ such that $|ta| \geq \Bar{t}>0$ for all $t> T$ and 
$$
\frac{1}{\kappa}\int_{\Bar{t}}^\infty |g(t)| \dd t < \frac{1}{a}\int_{\Bar{t}}^\infty |g(t)| \dd t < \frac{\varepsilon} {2} .
$$
Let $t \in (T,\infty)$ and $n_1 \in \N$ such that for all $n\geq n_1$, we have $|\kappa -\kappa_n |  \frac{2||g||_{1;[0,\infty)}}{a^2} < \varepsilon/2$. Then,
\begin{align*}
    |H_n(t) -H(t)| &= \left|\frac{1}{\kappa}\int_{0}^{t\kappa} g(s) \dd s - \frac{1}{\kappa_n}\int_{0}^{t\kappa_n} g(s) \dd s\right| \\
    &= \left|\frac{1}{\kappa}\int_{0}^{t\kappa} g(s) \dd s - \left(\frac{1}{\kappa} + \frac{\kappa -\kappa_n}{\kappa \kappa_n} \right)\left(\int_{0}^{t\kappa} g(s) \dd s + \int_{t\kappa}^{t\kappa_n} g(s) \dd s \right) \right| \\
    &\leq \frac{1}{\kappa} \left|\int_{t\kappa}^{t\kappa_n}g(s) \dd s\right| + \frac{\left|\kappa -\kappa_n\right|}{\kappa \kappa_n} \left|\left(\int_{0}^{t\kappa} g(s) \dd s + \int_{t\kappa}^{t\kappa_n} g(s) \dd s \right)\right| \\
    &\leq |\kappa -\kappa_n |  \frac{2||g||_{1;[0,\infty)}}{a^2} + \frac{1}{\kappa}
    \left|\int_{t\kappa}^{t\kappa_n} g(s) \dd s\right| \\
    & \leq |\kappa -\kappa_n |  \frac{2||g||_{1;[0,\infty)}}{a^2} + \frac{1}{\kappa} \int_{\Bar{t}}^\infty |g(s)| \dd s < |\kappa -\kappa_n |  \frac{2||g||_{1;[0,\infty)}}{a^2} + \frac{\varepsilon}{2} < \varepsilon.
\end{align*}
For $t\in [0,T]$ recall that $H_n$ is equicontinuous and uniformly bounded, hence, by Arzel{\`a}-Ascoli, there exists a subsequence $n_k$ such that $H_{n_k} \to H$ uniformly on $[0,T]$, i.e.~there exists $k_2$ such that for all $k\geq k_2$, we have
$$
    ||H_{n_k} - H||_{\infty; [0,T]} < \varepsilon.
$$
By changing to a subsequence for $t>T$ and replacing $n_1$ by $k_1$, we find that there exists a subsequence $n_k$ such that for all $k\geq \max\{k_1,k_2\}$, we have
$$
||H_{n_k} - H||_{\infty;\R} < \varepsilon,
$$
i.e.~$H_{n_k} \to H$ uniformly on $\R$. The statement for $H_\kappa^2$ follows from the fact that $H_\kappa$ is bounded by $1/a ||g||_1 + ||\alpha||_1$ for $\kappa \in [a,b]$ and therefore
$$
||H_{n_k}^2 - H_\kappa^2 ||_{\infty;\R} \leq ||H_{n_k} + H_\kappa ||_{\infty;\R} \ ||H_{n_k} - H_\kappa  ||_{\infty;\R} \to 0.
$$
\end{proof}

The following is a slightly adapted version of \cite[Thm.~3.3]{ode_levi_2017}.
\begin{theorem}\label{thm::comparison}
    Let $f,g : \R^+ \times \R \to \R$ satisfy \autoref{cond::1}. Consider the two ODEs
    \begin{align*}
        & \dot{x} = f(t,x) \qquad x(0) = x_0\\
        & \dot{y} = g(t,y) \qquad y(0) = y_0.
    \end{align*}
    Assume $f(t,x) \leq g(t,x)$ for all $x\in\R$ and $x_0 < y_0$. Then $x(t) < y(t)$ for all $t\geq 0$ in the maximal interval of existence for both solutions.
\end{theorem}
\begin{proof}
    Let $\varepsilon>0$ and $y_1, y_1^\varepsilon$ be the solution to 
    \begin{align*}
    & \dot{y_1} = g(t,y_1) \qquad y_1^\varepsilon(0) = x_0 + \frac{y_0-x_0}{2} \\
    & \dot{y_1^\varepsilon} = g(t,y_1^\varepsilon) + \varepsilon \qquad y_1^\varepsilon(0) = x_0 + \frac{y_0-x_0}{2}.
    \end{align*}
    Since $f,g,g+\varepsilon$ all satisfy \autoref{cond::1} there exist unique solutions to all here considered ODEs (\cite{odetheory}). \\
    Assume that there exists $\Bar{t}$ such that $x(\Bar{t}) < y_1^\varepsilon(\Bar{t})$. By the continuity of $x,y_1^\varepsilon$ there exists $t_1\geq 0$ such that $x(t_1) = y_1^\varepsilon(t_1)$ and $x(t) < y_1^\varepsilon(t)$ for all $t\in(t_1,\Bar{t})$. However, this yields a contradiction as $\dot{x}(t_1) = f(t_1,x) < g(t_1,y^\varepsilon_1) + \varepsilon = \dot{y_1^\varepsilon}$. Hence, $x\leq y_1^\varepsilon$ for all $\varepsilon > 0$. \\
    Let $T>0$ and consider $z \coloneqq y_1^\varepsilon - y_1$, then 
    \begin{align*}
        |z'(t) | = |g(t, y_1^\varepsilon) + \varepsilon- g(t, y_1)| \leq \varepsilon + ||g'||_{_{\infty;[0.T]}} |z|.
    \end{align*}
    By Gronwall's inequality we have $|z| \leq \varepsilon e^{||g'||_{_{\infty;[0.T]}} T}$. And hence $y_1^\varepsilon \to y_1$ uniformly on compact intervals as $\varepsilon\to 0$. In particular, $y_1^\varepsilon \to y_1 $ pointwise and hence $y_1 \geq x$. Finally, note that by uniqueness of solutions $y_1 < y$ and therefore $x < y$. 
\end{proof}

\vspace{0.75cm}

\begin{proof}[Proof of \autoref{lem::propertieslossfunc}]
    (ii) Let $I=[\kappa_-,\kappa_+) \subset \R^+$ and note that for $\kappa_1,\kappa_2 \in I$
    \begin{align}\label{eqq::proof::normlowerbound}
    \begin{split}
    || g(\cdot \kappa_1) - g(\cdot \kappa_2)||_{2;[0,\infty)} &= || \int_{\kappa_1}^{\kappa_2} \cdot g'(\cdot \kappa) \dd \kappa ||_{2;[0,\infty)} \leq \int_{\kappa_1}^{\kappa_2} ||\cdot g'(\cdot\kappa) ||_{2;[0,\infty)}   \dd \kappa \\
    &\leq |\kappa_2 - \kappa_1| \sup_{\kappa \in [\kappa_-,\kappa_+)} ||\cdot g'(\cdot\kappa)||_{2;[0,\infty)} \\
    & =|\kappa_2 - \kappa_1|\sup_{\kappa \in [\kappa_-,\kappa_+)} \frac{1}{\kappa^{3/2}} ||\cdot g'(\cdot)||_{2;[0,\infty)} = |\kappa_2 - \kappa_1| \frac{1}{\kappa_-^{3/2}} ||\cdot g'(\cdot)||_{2;[0,\infty)},
    \end{split}
    \end{align}
    where we used the Minkowski integral inequality for the first inequality.
Let $u = \alpha - g(\cdot \kappa_1)$ and $v = \alpha - g(\cdot \kappa_2)$, then 
\begin{align}\label{eqq::proof::normdecomp}\begin{split}
    \bigg|||\alpha - g(\cdot \kappa_1)||_{2;[0,\infty)}^2 &- ||\alpha - g(\cdot \kappa_2)||_{2;[0,\infty)}^2 \bigg| 
    = | \langle u,u\rangle - \langle v,v\rangle | 
    = | \langle u,u\rangle - \langle u,v \rangle + \langle u,v\rangle -\langle
    v,v\rangle | \\
    & = |\langle u+v, u-v \rangle| 
    \leq ||u+v||_{2;[0,\infty)} ||u-v||_{2;[0,\infty)} \\
    & \leq \bigg(||\alpha - g(\cdot \kappa_1)||_{2;[0,\infty)} + ||\alpha - g(\cdot \kappa_2)||_{2;[0,\infty)} \bigg) ||g(\cdot \kappa_1) - g(\cdot \kappa_2)||_{2;[0,\infty)},
    \end{split}
\end{align}
where we used Cauchy Schwartz for the first inequality. Naturally for $i=1,2$
\begin{align*}
    ||g(\cdot \kappa_i)||_{2;[0,\infty)} = \frac{1}{\sqrt{\kappa_i}} ||g||_{2;[0,\infty)} \leq \frac{1}{\sqrt{\kappa_-}} ||g||_{2;[0,\infty)}.
\end{align*}

Hence, \eqref{eqq::proof::normlowerbound} and \eqref{eqq::proof::normdecomp} yield
\begin{align*}
   \bigg|||\alpha - g(\cdot \kappa_1)||_{2;[0,\infty)}^2 &- ||\alpha - g(\cdot \kappa_2)||_{2;[0,\infty)}^2 \bigg| \\ 
   &\leq  \bigg(||\alpha - g(\cdot \kappa_1)||_{2;[0,\infty)} + ||\alpha - g(\cdot \kappa_2)||_{2;[0,\infty)} \bigg) ||g(\cdot \kappa_1) - g(\cdot \kappa_2)||_{2;[0,\infty)} \\
   &\leq \bigg(2||\alpha||_{2;[0,\infty)} + 2 \frac{1}{\sqrt{\kappa_-}} ||g||_{2;[0,\infty)} \bigg) ||g(\cdot \kappa_1) - g(\cdot \kappa_2)||_{2;[0,\infty)} \\
   & \leq  \bigg(2||\alpha||_{2;[0,\infty)} + 2 \frac{1}{\sqrt{\kappa_-}} ||g||_{2;[0,\infty)} \bigg) \frac{1}{\kappa_-^{3/2}} ||\cdot g'(\cdot)||_{2;[0,\infty)} |\kappa_2 - \kappa_1| \\
   & = L |\kappa_2 - \kappa_1|
\end{align*}
for $L \coloneqq \bigg(2||\alpha||_{2;[0,\infty)} + 2 \frac{1}{\sqrt{\kappa_-}} ||g||_{2;[0,\infty)} \bigg) \frac{1}{\kappa_-^{3/2}} ||\cdot g'(\cdot)||_{2;[0,\infty)} $.
\\\\
(i) Let $\kappa_n \to \kappa$, then there exists $n_0\in \N $ and a compact interval $I \subset \R^+$ such that $\kappa_n \in I$ for all $n\geq n_0$. By (ii) the claim holds for all $f \in C^\infty_c$ the set smooth functions with compact support as $f$ is naturally of bounded variation with $tf'\in L^2$. Since $C^\infty_c$ is dense in $L^2$, we may take a function $g_k \in C^\infty_c$ with $g_k \to g$ in $L^2$. Then
\begin{align*}
    &\bigg| ||\alpha - g(\cdot \kappa_n)||_{2;[0,\infty)} - ||\alpha - g(\cdot \kappa)||_{2;[0,\infty)}\bigg|
    \leq \bigg| ||\alpha - g(\cdot \kappa_n)||_{2;[0,\infty)} - ||\alpha - g_k(\cdot \kappa_n)||_{2;[0,\infty)}\bigg| \\
    & \qquad \qquad + \bigg| ||\alpha - g_k(\cdot \kappa_n)||_{2;[0,\infty)} - ||\alpha - g_k(\cdot \kappa)||_{2;[0,\infty)}\bigg| 
    +\bigg| ||\alpha - g_k(\cdot \kappa)||_{2;[0,\infty)} - ||\alpha - g(\cdot \kappa)||_{2;[0,\infty)}\bigg| \\
    &\leq || g(\cdot \kappa_n) - g_k(\cdot \kappa_n)||_{2;[0,\infty)}
    + || g_k(\cdot \kappa_n) - g_k(\cdot \kappa)||_{2;[0,\infty)}
    + || g_k(\cdot \kappa) - g(\cdot \kappa)||_{2;[0,\infty)}\longrightarrow 0 \qquad n,k \to \infty,
\end{align*}
which proves the claim.
\end{proof}

\begin{lemma}\label{lem::integral1nu}
    Let $\lambda>0$ and $\alpha \in \SET(\R)$ with $\alpha \geq 0$. Consider the equation
    \begin{align}\label{eqq::proof::nu_statement}
        \dot{x}(t) = \lambda - \alpha(t) - x^2(t)
    \end{align}
    and assume that $\alpha$ is such that $\attr^x(0) >-\sqrt{\lambda}$. Moreover, let $m_{\lambda, \alpha} = \sqrt{\lambda} + \attr^x(0)$ and assume that \eqref{eqq::proof::nu_statement} is in Case C (i.e.~$\attr^x$ diverges to $-\infty$ cf.~\autoref{thm::casesx}). Then there exists $\nu \geq0$ such that 
    \begin{align*}
        \int_0^\infty  \big(\alpha(sm_{\lambda, \alpha}) - \frac{\nu}{2}\big)^+ \dd s = 1.
    \end{align*}
\end{lemma}

\begin{proof}
Define
\begin{align*}
    \xi(\nu) \coloneqq \int_0^\infty  \big(\alpha(sm_{\lambda, \alpha}) - \frac{\nu}{2}\big)^+ \dd s 
\end{align*}
Naturally, $\xi(2||\alpha||_\infty) = 0$, if furthermore $\xi(0) = \int_0^\infty (\alpha(sm_{\lambda, \alpha}))^+ \dd s \geq 1$, then the intermediate value theorem shows that there exists $\nu$ with $\xi(\nu) =1$. Assume that $\xi(0) < 1$, then
\begin{align*}
    1 > \xi(0) = \int_0^\infty (\alpha(sm_{\lambda, \alpha}))^+ \dd s = \int_0^\infty \alpha(sm_{\lambda, \alpha}) \dd s = \frac{||\alpha||_{1;[0,\infty)}}{m_{\lambda, \alpha}}.
\end{align*}
Consider the equation
\begin{align}\label{eqq::proof::sec4uncontrolled}
    \dot{x}(t) = \lambda - \alpha(t\kappa) - x(t)^2 .
\end{align}
By \autoref{thm::largekappa} (ii), \eqref{eqq::proof::sec4uncontrolled} is in Case A for all $\kappa > \frac{||\alpha||_{1;[0,\infty)}}{m_{\lambda, \alpha}}$. However, $ 1> \frac{||\alpha||_{1;[0,\infty)}}{m_{\lambda, \alpha}}$, so in particular the uncontrolled equation
\begin{align*}
    \dot{x}(t) = \lambda - \alpha(t) - x^2(t)
\end{align*}
is in Case A. This contradicts the assumption that the uncontrolled equation is in Case C.
\end{proof}

\section*{Data availability}
No data was used for the research described in the article.

\section*{Acknowledgements}
Kerstin Lux-Gottschalk and Christopher Beekmann received funding via the Irène Curie fellowship of the
Eindhoven University of Technology.

\section*{Declaration of generative AI and AI-assisted technologies in the manuscript preparation process}

During the preparation of this work, the authors used Chat GPT-5 in order to gain a better overview on some tools from adjacent fields as well as for insights on some continuity aspects of the rate function. After using this tool, the authors reviewed and edited the content as needed and take full responsibility for the content of the published article.

\printbibliography[title={References}]
\end{document}